\documentclass[DIV=10]{scrartcl}

\usepackage{comment}
\usepackage{subfig}
\usepackage{tabularx}
\usepackage{makecell}
\usepackage{amsfonts}
\usepackage{amsfonts, amsmath, amssymb}
\usepackage{amsthm}
\usepackage{csquotes}
\usepackage{enumitem}
\usepackage{mathtools}
\usepackage[%
    hyperfootnotes=false,
    colorlinks = true,
    final=true,
    plainpages=false,
    pdfstartview=FitV,
    pdftoolbar=true,
    pdfmenubar=true,
    pdfencoding=auto, 
    psdextra, 
    bookmarksopen=true,
    bookmarksnumbered=true,
    breaklinks=true,
    linktocpage=true,
    citecolor=blue,
    linkcolor=blue]%
    {hyperref}
\usepackage[nameinlink]{cleveref}
\Crefname{figure}{}{}

\usepackage[style=numeric-comp,%
			defernumbers=true,
			useprefix=true,%
			giveninits=true,%
			hyperref=true,%
			uniquename=init,%
			sorting = none,
			sortcites=false,
			doi=true,%
			isbn=false,%
			url=false,%
			backend=biber%
	]{biblatex}
\usepackage{graphicx}
\usepackage{titling}
\usepackage{authblk}

\newtheorem{theorem}{Theorem}
\newtheorem{lemma}[theorem]{Lemma}
\newtheorem{cor}[theorem]{Corollary}
\newtheorem{assumption}{Assumption}
\newtheorem{remark}{Remark}%

\newtheorem{definition}{Definition}%

\newcommand{\R}{{\mathbb R}}

\newcommand{\N}{{\mathbb N}}

\newcommand{\expect}{\operatorname{\mathbb{E}}}

\DeclareMathOperator*{\argmin}{argmin}

\DeclareMathOperator{\sign}{\text{sign}}

\addtokomafont{disposition}{\rmfamily}

\makeatletter
\let\blx@rerun@biber\relax
\makeatother

\title{Explainable Learning Based Regularization of Inverse Problems}

\author[1, 2]{Martin Burger}
\affil[1]{Helmholtz Imaging, Deutsches Elektronen-Synchrotron DESY, Notkestr. 85, Hamburg, 22607,
Germany,}
\affil[ ]{\{martin.burger, samira.kabri, lukas.weigand\}@desy.de}
\affil[2]{Fachbereich Mathematik, Universit\"at Hamburg, Bundesstrasse 55, Hamburg, 20146, Germany}
\author[1]{Samira Kabri}
\author[3]{Gitta Kutyniok}
\affil[3]{Lehrstuhl für mathematische Grundlagen des Verständnisses der künstlichen \hphantom{Akadem} Intelligenz, Mathematisches Institut der Universität München, \hphantom{Akadem} Akademiestraße 7, 80799 München, Germany,}
\affil[ ]{kutyniok@math.lmu.de}
\author[4]{Yunseok Lee}
\affil[4]{yu2792304@gmail.com}
\author[1]{Lukas Weigand}
\date{}
\begin{document}

\maketitle

\abstract{Machine learning techniques for the solution of inverse problems have become an attractive approach in the last decade, while their theoretical foundations are still in their infancy. In this chapter we want to pursue the study of regularization properties, robustness, convergence rates, and structure of regularizers for inverse problems obtained from different learning paradigms. For this sake we study simple architectures that are explainable in the sense that they allow for a theoretical analysis also in the infinite-dimensional limit. 
In particular we will advance the study of spectral architectures with new results on convergence rates highlighting the role of the smoothness in the training data set, and a study of adversarial robustness. We can show that adversarial training is actually a convergent regularization method. Moreover, we discuss extensions to frame systems and CNN-type architectures for variational regularizers, where we obtain some results on their structure by carefully designed numerical experiments.}
\section{Introduction} 

Data-driven methods for the solution of inverse problems have gained strong attention in recent years due to the {  vast} advances made in (deep) machine learning (cf. e.g. \cite{arridge2019solving_MB}). Compared to other fields, inverse problems pose particular challenges due to their ill-posedness, complexity, and the absence of ground truth data for measurements and inverse solutions. For these reasons, the straightforward application of supervised learning techniques seems still out of reach for most practical applications. Semi-supervised learning approaches have emerged to an attractive alternative approach, which {  typically learn} a regularizer or some related features reminiscent of a proximal map of regularizers. Those are based either solely on suitable inverse solutions or non-favorable inverse solutions, e.g. obtained from unstable reconstructions.  

Despite some practical success of these approaches,  theoretical insights such as explanations of results achieved, an analysis of their robustness, or a thorough convergence analysis in terms of statistical regularization are hardly available. In this chapter we want to further discuss these theoretical issues by studying simple explainable architectures, which allow for a detailed analysis and also for some comparison with classical approaches in the regularization of inverse problems. In particular we shall review some results on spectral architectures related to the classical technique of singular value decomposition (cf. \cite{Kabri2023Convergent_MB, Burger2025Learned_MB}) and augment them by novel results such as convergence rates of learned regularizations. We also discuss how the spectral architectures and the theoretical insights can be generalized to frame systems. As we shall see, these architectures allow to derive deep theoretical insights into the regularization properties of learned methods, their robustness, and effects arising in the infinite-dimensional limit relevant in inverse problems. 

In addition we will also study adversarial examples in the framework of spectral architectures, which we understand in the sense of adversarial data (distributions) for the inverse problems. As mentioned above one has to deal with the issue of distribution shifts in these data in particular for supervised learning techniques and hence it is of particular importance to study the most severe effects on robustness stemming from single or population data drawn from a perturbed distribution. 

As a step towards practically used convolutional neural network (CNN) architectures, we finally study a CNN type approach mimicking multiresolution or total variation type regularization. We construct some computational studies that tackle the questions whether nonlocality of convolutional kernels is favorable and with which types of CNN kernels we genuinely end up in a more conventional setting of local kernels. While this chapter focuses on explainability through simple architectures, we refer to the open source python library DeepInverse \cite{tachella2025deepinverse_MB} for advanced deep neural network architectures for solving imaging inverse problems.

Before proceeding to the study of explainable architectures we shall give a short introduction to different learning paradigms for inverse problems in the next section.

\section{Learning paradigms for regularization methods in inverse problems}

Inverse problems are of central importance in many fields of science and technology, with image reconstruction in different modalities being a prominent example.  In general we are interested in abstract equations of the form
\begin{equation}
    A(x) = y,
\end{equation}
with $A$ being a forward operator whose inverse does not exist or is not continuous. The latter causes obvious difficulties in the solution process in particular when noisy data $y^\delta$ are available instead of $y$ (the superscript $\delta$ denoting a typical magnitude of the noise). The classical approach consists in constructing handcrafted regularization methods based on mathematical principles (cf. \cite{engl1996regularization_MB,benning2018modern_MB}). Recently a data driven construction of such methods received strong attention due to the increasing availability of potential training data and advances in deep learning implementations allowing to tackle larger scale problems. Due to the ill-posedness in the solution of inverse problems and the need to introduce a-priori knowledge anyway, such approaches appear rather natural and promising. In the following we give a high level overview of the different approaches, which we group into three different groups according to the paradigm of what to learn rather than by specific architectures how to do it.  

\subsection{Supervised learning}

The basis of supervised learning approaches is the availability of data pairs $(x_i,y_i^\delta)_{i=1,\ldots,N}$, from which the inverse solution is to be generalized. Since the training data ideally involves noisy sample this can be conceived as learning a regularization method, in particular if the learning is performed conditional on the noise strength $\delta$. With appropriate architectures $R_\theta(y)$ depending on learnable parameters $\theta$ one tries to parametrize the regularized inverse solution. As usual in supervised learning in arbitrary applications this is achieved by minimizing an overall loss of the form 
\begin{equation}
L(\theta) = \frac{1}N \sum_{i=1}^N \ell(R_\theta(y_i^\delta),x_i).
\end{equation}
Here $\ell$ denotes a suitable distance between the computed solution $R_\theta(y_i^\delta)$ and the corresponding training data $x_i$.

Different architectures have been investigated for the regularization method $R_\theta$. The most straightforward approach is to try standard architectures for deep neural networks such as U-nets directly approximating a regularized inverse $R_\theta$. Such approaches encounter limitations in many inverse problems in practice due to the complexity of the forward operator (and thus the approximation of the full inversion). Moreover, the generalization properties are doubtful in case of ill-posed problems, since the networks necessarily need to have a large Lipschitz constant, for which the usual machine learning theory would predict an unreasonable generalization error. 

In order to cure the issues in direct end-to-end learning mentioned above, several approaches that include information about the forward operator in the architecture  have been considered. An early and frequently used approach in this direction is to use some approximate inverse $B$ of $A$ and choose the architecture $R_\theta$ of the form
\begin{equation}
    R_\theta(y) = {\cal N}_\theta(B(y)), 
\end{equation}
where ${\cal N}_\theta$ is an appropriate deep neural network. In this way one transfers the inversion problem to a learning problem in the domain of the inverse solutions. 
Then the neural network merely functions as post processing, learning to denoise and correct the rough reconstruction $B(y)$ for artifacts.   
A typical example is (undersampled or limited angle) X-ray tomography, where $B$ can be chosen as a filtered backprojection operator and the subsequent learning arises between images, where a variety of highly developed CNN-type architectures are available (cf. \cite{chen2017low_MB,mccann2017convolutional_MB}).

An alternative approach is guided by well-understood iterative methods for inverse problems such as proximal algorithms, which are then augmented by learnable parameters in order to achieve a reasonable convergence and regularization in a fixed number of iterations (corresponding to layers in deep networks). This approach is typically called {\em unrolling} and comprises several different iterative techniques as well as different types of parametrization, e.g. with fixed or with changing parameters in the different iterations. From a deep learning perspective unrolling can be interpreted as residual networks involving $A$, the adjoint of its derivative and a skip connection to the input layer (the data $y$) (cf. \cite{li2019algorithm_MB}).

\subsection{Unsupervised learning with adversarial data}

The second learning paradigm arises from data sets where both data $(x_i)_{i=1,\ldots,N}$ about the favorable solutions as well as noisy measurement  data $(y_j^\delta)_{j=1,\ldots,M}$ are available, but not in the form of corresponding pairs. Given some approximate inverse $B$ as before, one can generate adversarial data $x_j^\delta = B(y_j^\delta)$, i.e. unfavorable inverse solutions. 

The most prominent example is the learning of adversarial regularizers in a variational setting (cf. \cite{lunz2018adversarial_MB}). The key idea is to learn a parametrized regularization functional $J_\theta$ for a variational regularization method, i.e., 
\begin{equation}
    R_\theta(y) = \text{arg}\min_x \left( D(A(x),y)+\alpha J_\theta(x) \right), 
\end{equation}
with $D$ being an appropriate distance in the measurement domain and $\alpha >0$ a regularization parameter.  Adversarial regularizers are obtained by minimizing the loss 
\begin{equation}
    L(\theta) = \frac{1}N \sum_{i=1}^N J_\theta(x_i)  - \frac{1}M \sum_{j=1}^M J_\theta(x_j^\delta) + \mathcal{S}(\nabla J_\theta), 
\end{equation}
with $\mathcal{S}$ being an appropriate functional penalizing large gradients. Note that the different signs in front of the first two terms favors optimal solutions for $J_\theta$ to take small values on the favorable data $x_i$ and large ones on the adversarial data $X_j^\delta$. The penalization of gradients is necessary to obtain meaningful regularization functionals. By analogy to the dual definition of the $1$-Wasserstein metric between the favorable and adversarial data populations a hard constraint for the gradient has been proposed in \cite{lunz2018adversarial_MB}, i.e.,
\begin{equation}
\mathcal{S}(\nabla J_\theta) = \left\{ \begin{array}{ll} 
0 & \text{if } \Vert \nabla J_\theta(x) \Vert \leq 1 \text{ almost everywhere,} \\ + \infty &\text{else.} \end{array} \right.
\end{equation}
However, it has been shown later that a good approximation of the Wasserstein metric is not desirable anyway (cf. \cite{stanczuk2021wasserstein_MB}), hence other standard penalization functionals can be used, which might increase interpretability of solutions (cf. \cite{mukherjee2021learning_MB,sriram_MB}). Similar adversarial learning techniques can also be coupled with unrolling of iterative algorithms (cf. \cite{mukherjee2021end_MB})

\subsection{Unsupervised generative learning}

The third class of approaches does not use any kind of measurement data for training, but solely data $(x_i)_{i=1,\ldots,N}$  that correspond to typical or desired solutions for the inverse problem, e.g. just typical images in the case of image reconstruction.  
What is to be learned from this type of data is somehow a projection to the manifold of favorable solution or some vector field pointing to it (or in a stochastic interpretation maps or velocities towards higher probability density). Alternatively, these approaches can be understood as generative models, since via projections or velocity fields one is able to generate favorable solutions from arbitrary inputs. 

A canonical and thus frequent choice are denoising models (cf. \cite{Romano2017Little_MB,Reehorst2019Regularization_MB}), which learn a denoising scheme ${\cal D}_\theta(x)$. The simplest version to learn a denoising scheme is supervised learning from artificially generated noisy samples, e.g. by adding random noise to the given samples $x_i$. A more recent approach are denoising diffusion models, which learn a network approximating the score of a diffusion process (related to adding Gaussian noise of increasing strength) instead (cf. \cite{Song2021Score_MB}). These networks can be used as velocity fields towards the manifold of favorable solutions and different -- stochastic or deterministic -- flows back to this manifold can be constructed from it. 

In order to employ denoising diffusion models for inverse problems several methods have been proposed, often under the umbrella of so-called plug-and-play methods (cf. \cite{venkatakrishnan2013pnp_MB,rond2016poisson_MB,kamilov2017plug_MB}). Again, the most straight-forward approach is to use denoising as a postprocessing after a preliminary inversion, i.e.
\begin{equation}
    R_\theta(y) = {\cal D}_\theta(B(y)),
\end{equation}
where ${\cal D}_\theta$ is a parametrized denoiser. A more advanced approach consists of including the denoiser in an iterative scheme. Particularly attractive are schemes resembling proximal splitting in variational regularization (cf. \cite{meinhardt2017learning_MB,wei2022tfpnp_MB}). With a distance measure $D$ as above, the most prominent example is of the form $R_\theta(y) = X_K$, with 
\begin{equation}
    x_{k+1} = {\cal D}_\theta ( x_k - \tau \nabla_x D(A(x),y)), \qquad k=1,\ldots,K-1.
\end{equation}
In approaches providing a velocity field such as diffusion models, several different approaches have been proposed, which combine the velocity with the gradient of the distance $D(A(x),y)$, or use variational inference techniques to locally approximate a related Bayesian posterior (cf. e.g. \cite{Barbano2024Steerable_MB,Feng2023Score_MB,Feng2024Variational_MB,mardani2023variational_MB}). The theoretical justification of these approaches is however widely open.

While diffusion models rely on the formalism of a stochastic differential equation to transform a simple to sample distribution $P_1$, e.g. a normal distribution, into the hard to sample distribution $P_2$ of typical solutions , this can be also be done via the deterministic continuity equation. 
The velocity field corresponding to the continuity equation can be parametrize by a neural network and samples from distributions $P_1$ can be pushed onto samples from $P_2$ by solving an ODE. This approach has been utilized in Flow Matching and Continuous Normalizing Flows (cf. \cite{wald2025flow_MB}), which provide different tractable loss functions to learn such a velocity field.
In the context of Bayesian inverse problems Flow Matching on conditional Wasserstein spaces has been propose to sample from the posterior (cf. \cite{chemseddine2025conditional_MB}).

\section{Learning-based spectral regularizations}

We are interested in inverse problems of the form 
\begin{equation}\label{mb_eq:invproblem}
    y^{\delta} = Ax + \varepsilon,
\end{equation}
where $y^\delta$ is a measurement, that is obtained from an unknown quantity $x$ with a linear, compact forward operator between Hilbert spaces $A: X \rightarrow Y$ and corrupted by unknown noise $\varepsilon$. 
The severeness of the noise is measured by the so-called noise level, denoted by the superscript $\delta$, and is assumed to be controllable.

Since $A$ is compact and linear, it admits a singular value expansion $( u_n, v_n, \sigma_n )_{n\in\N}$, such that 
\begin{equation}\label{mb_eq:svd}
    Ax = \sum_{n \in \N} \sigma_n \langle x,u_n \rangle v_n.
\end{equation}
A possible solution operator for uncorrupted data is the pseudo-inverse $A^\dagger: \mathcal{R}(A)\oplus\mathcal{R}(A)^\perp \rightarrow \mathcal{N}(A)^\perp$,
\begin{equation}\label{mb_eq:pseudoinv}
    A^\dagger y = \sum_{n \in \N} \frac{1}{\sigma_n} \langle y,v_n \rangle u_n.
\end{equation}
While using $A^\dagger$ as a reconstruction operator circumvents the problems of non-existence and non-uniqueness, it might be an instable reconstruction operator. This means, that error the reconstruction error can not be controlled by the measurement error caused by the noise. The aim of regularization is to reconstruct noisy measurements with smoothened versions of the pseudo-inverse (or any other desirable but instable solution operator) while still approaching the pseudo-inverse as the noise vanishes. 

\subsection{Classical regularization theory} \label{ch:_ClassRegMB}
Classical regularization theory considers reconstruction operators $R_\alpha$ parametrized by a regularization parameter $\alpha \in \R_{>0}$. The choice of this parameter depends on the, {  to reconstruct}, measurement $y^\delta$  and/or the noise level $\delta$. In the following, we recapitulate some foundations of the classical theory presented in \cite{engl1996regularization_MB}.
\begin{definition}[{Regularization}]
    Given $\alpha_0 \in \R_{>0}$ we call the family $\lbrace R_\alpha \rbrace_{\alpha \in (0, \alpha_0)}$ a  \emph{regularization} if  $R_\alpha: Y \rightarrow X$ is continuous for any $\alpha \in (0,\alpha_0) $.
\end{definition}
\begin{definition}[{Convergence}]\label{def: Reg_MB}
    A regularization $\lbrace R_{\alpha} \rbrace_{\alpha \in (0, \alpha_0)}$ is called \emph{convergent} if for any $y\in \mathcal{D}(A^\dagger)\subset Y$ there exists a parameter choice rule $\alpha: \R_{>0} \times Y \rightarrow (0, \alpha_0)$ such that 
    \begin{equation}\label{mb_eq:classicallimsup}
    \lim_{\delta \rightarrow 0}\sup \left\{ \| R_{\alpha(\delta,y^\delta)} y^\delta-A^\dagger y\| \,\Bigl | \;y^\delta \in Y , \|y-y^\delta\|\leq \delta \right\}=0. 
    \end{equation}
\end{definition}
We note that for a fixed $y\in \mathcal{D}(A^\dagger)$ Equation \eqref{mb_eq:classicallimsup} implies uniform convergence over all admissible $y^\delta$. However this convergence is not uniform over all $y\in \mathcal{D}(A^\dagger)$, but merely point wise. 
It is shown in \cite[Proposition 3.11]{engl1996regularization_MB} that for forward operators $A$ with an unbounded pseudo-inverse $A^\dagger$ there cannot exist a linear regularization method such that \eqref{mb_eq:classicallimsup} converges uniformly for $y\in \mathcal{D}(A^\dagger)$, $\|y\| \leq 1$. This can be explained with the modulus of continuity of the operator $A$. It is defined for $\delta > 0$ on a set $M \subset X$ as
\begin{equation*}
    \Omega(\delta,M) := \sup\{\|x\|\,|\; x\in M, \|Ax\|\leq \delta \}
\end{equation*}
and is a lower bound for the error in \eqref{mb_eq:classicallimsup} (see \cite[Proposition 3.10]{engl1996regularization_MB}) in the sense that
\begin{equation*}
    \sup \left\{ \| R y^\delta-x\| \,\Bigl | \;x \in M, y^\delta \in Y , \|Ax-y^\delta\|\leq \delta \right\} \geq \Omega(\delta,M).
\end{equation*}
Choosing $M = \mathcal{N}(A)^\perp \cap T^{-1}(B_1(0))$ we {  obtain} that for any $x \in M$, $y^\delta \in Y$ that fulfill $\|Ax-y^\delta\| \leq \delta$, it holds that $y := Ax \in \mathcal{D}(A^\dagger)\cap B_1(0)$ and $A^\dagger y = x$. Therefore 
\begin{align*}
    \sup \left\{ \| R_{\alpha(\delta,y^\delta)} y^\delta-A^\dagger y\| \,\Bigl | \; y \in \mathcal{D}(A^\dagger)\cap B_1(0), y^\delta \in Y , \|y-y^\delta\|\leq \delta \right\}& \\
    \geq \Omega(\delta,  \mathcal{N}(A)^\perp \cap T^{-1}(B_1(0))),
\end{align*}
where the right-hand side is infinity if $A^\dagger$ is unbounded. Since regularization of a bounded pseudo-inverse is not of major interest, convergence rates are usually stated for measurements generated from ground truth data contained in a smaller set $M$. These sets are then called source sets and a popular choice are the sets
\begin{equation*}
    \chi_{\mu,\rho} := \left \{ x \in X \, \Bigl | \;x = (A^*A)^{\mu}\,w, \, \|w\| \leq \rho\right\},
\end{equation*}
for which it holds that (see \cite[Proposition 3.14]{engl1996regularization_MB})\begin{equation*}
    \Omega(\delta, \chi_{\mu,\rho}) \leq \delta^{\frac{2\mu}{2\mu + 1}}\rho^{\frac{1}{2\mu+1}}.
\end{equation*}
It is further shown in \cite[Proposition 3.15]{engl1996regularization_MB} that this bound is sharp. An example for a regularization that obtains convergence rate of optimal order $\delta^{\frac{2}{3}}$ on the source sets $\chi_{1,\rho}$  is Tikhonov regularization (see \cite[Chapter 5.1]{engl1996regularization_MB}). For a regularization parameter $\alpha > 0$, the Tikhonov regularizer can be defined implicitly by a variational problem,
\begin{equation*}
R_{\alpha} y^\delta = \argmin_{x \in X} \frac{1}{2}\|Ax-y^\delta\|^2 + \frac{\alpha}{2} \|x\|^2,
\end{equation*}
or explicitly via the singular value expansion of $A$,
\begin{equation*}
    R_{\alpha} y^\delta = \sum_{n \in \N} \frac{\sigma_n}{\sigma_n^2 + \alpha} \langle y^\delta, v_n\rangle u_n,
\end{equation*}
This is a special case of regularization by spectral filtering of the pseudo-inverse $A^\dagger$, 
\begin{equation*}
    R_{\alpha} y^\delta = \sum_{n \in \N} g_\alpha(\sigma) \langle y^\delta, v_n\rangle u_n,
\end{equation*}
where $g_\alpha: \R_{>0}\rightarrow \R_{\geq0}$ is a filter function. In the following, we investigate different ways to obtain such spectral regularizers in data-driven settings.

\subsection{Data-driven spectral regularization}
In the setting of data-driven reconstruction, we are interested in the statistical version of the inverse problem \eqref{mb_eq:invproblem}. Thus, we assume our unknown ground truth data $x \sim P_x$ to be sampled from a data distribution and noise $\varepsilon \sim D_{\varepsilon}$ that is distributed independent of $x$ and sampled from a noise distribution.

Based on the singular value decomposition of $A^\dagger$ \eqref{mb_eq:pseudoinv}, the works \cite{Kabri2023Convergent_MB, Burger2025Learned_MB} study linear spectral reconstruction operators of the form 
\begin{equation}\label{mb_eq:spectralreg}
    R[g](y^{\delta})=  \sum_{n \in \N} g_n \langle y^{\delta},v_n \rangle u_n,
\end{equation}
with regularization coefficients $g = \left( g_n\right)_{n\in\N}$ that emerge from a data-driven optimization approach. In particular, \cite{Kabri2023Convergent_MB} focuses on the supervised learning approach 
\begin{equation}\label{mb_eq:mseobjective}
    g^{\text{MSE}}_{P_{\varepsilon}, P_x} \in \operatorname*{\arg \min}_{g} \expect_{x \sim P_x, \varepsilon \sim P_{\varepsilon}} \left[\left\| R(Ax+\varepsilon, g) - x\right\|^2\right] 
\end{equation} that was also studied in \cite{Chung2011Designing_MB} and optimizes the mean-squared error (MSE) between the input-output pairs $(Ax+\varepsilon, x)$ following the distributional laws $P_x$ and $P_{\varepsilon}$. Throughout this chapter, we assume that \begin{equation*}
    \expect_{x \sim P_{x}}\left[\|x\|^2\right] < \infty \qquad \text{and} \qquad  \expect_{\varepsilon \sim P_\varepsilon}\left[\varepsilon\right] = 0.
\end{equation*} With some further mild assumptions on the distributions with respect to the forward operator $A$ (see, e.g., \cite{Kabri2023Convergent_MB}) the minimizer of \eqref{mb_eq:mseobjective} is unique and given by 
\begin{equation}\label{mb_eq:gmse}
    g^{\text{MSE}}_{P_{\varepsilon}, P_x} = \left(\frac{\sigma_n \Pi_n}{\Pi_n\sigma_n^2 + \Delta_n}\right)_{n\in\N},
\end{equation}
where 
\begin{equation*}
    \Pi_n = \expect_{x \sim P_x}\left[\left|\langle x, u_n\rangle\right|^2\right] \qquad \text{and} \qquad \Delta_n = \expect_{x \sim P_{\varepsilon}}\left[\left|\langle \varepsilon, v_n\rangle\right|^2\right].
\end{equation*}
Here and in the following, the dependence of $\Pi_n$ and $\Delta_n$ on $P_x$ and $P_\varepsilon$ is omitted in the notation to improve readability.
The possibility to explicitly compute optimal parameters $g$ stems mainly from the orthogonality of the bases $(u_n)_{n\in\N}$ and $(v_n)_{n\in\N}$ and allows for the rigorous study of stability and convergence in different learning paradigms \cite{Burger2025Learned_MB}. To assess these properties, we use the definitions of data-driven regularizations introduced in \cite{Burger2025Learned_MB} that are based on the definitions proposed in \cite{benning2018modern_MB, engl1996regularization_MB}. 
For the sake of simplicity, in this chapter we only consider single-valued reconstruction operators $R$ and consider the pseudo inverse \eqref{mb_eq:pseudoinv} to be desired solution operator for uncorrupted data. We further restrict ourselves to the metric induced by the squared Hilbert-space norm that already appears in the formulation of the supervised approach \eqref{mb_eq:mseobjective}. 

Importantly, the understanding of parameters differs between the definitions of classic and data-driven regularizations. The regularization coefficients $g$ that appear in the definition of the spectral reconstruction operator R \eqref{mb_eq:spectralreg} are called regularization parameters in the classic formulation of regularization. In a data-driven setting these parameters are obtained automatically by an optimization process. Moreover, a data-driven approach relies on training examples drawn from training distributions. For example in \eqref{mb_eq:mseobjective}, the optimal parameters $g^{\text{MSE}}$ are automatically obtained by minimizing the MSE with respect to the training distributions $P_x$ and $P_\varepsilon$. Taking this into account, the definition of data-driven regularization treats these training distributions as what is called parameters in the classic setting.  Moreover, unless stated otherwise, we keep the data distribution $P_x$ fixed and denote by 
\begin{equation*}
R^{\text{MSE}}_{P_\varepsilon} := R\left[g_{P_\varepsilon, P_x}^{\text{MSE}}\right]
\end{equation*}
the data-driven reconstruction operator obtained by optimizing \eqref{mb_eq:mseobjective} for training distributions $P_x$ and $P_\varepsilon$. 
\begin{definition}[{Data-driven regularization}]
Given a family of noise distributions $\Psi$ we call the family $\lbrace R_{P_\varepsilon} \rbrace_{P_\varepsilon \in \Psi}$ a \emph{data-driven regularization} if  $R_{P_\varepsilon}: Y \rightarrow X$ is continuous for any choice of $P_\varepsilon \in \Psi$.
\end{definition}
The definition of convergence further requires a notion of the noise level $\delta \in \R^+$ of a noise distribution $P_\varepsilon$. In many cases, it can be defined via the function
\begin{equation*}
    \boldsymbol{\delta}(P_\varepsilon) = \sqrt{\expect_{\varepsilon \sim P_\varepsilon}\left[\left\|\varepsilon\right\|^2\right]}.
\end{equation*}
In \cite{Kabri2023Convergent_MB, Burger2025Learned_MB}, the noise level is defined via
\begin{equation}\label{mb_eq:noiselevelDelta}
    \boldsymbol{\delta}(P_\varepsilon) = \sqrt{\sup_{n \in \N}\expect_{\varepsilon \sim P_\varepsilon}\left[\left|\langle\varepsilon, v_n\rangle\right|^2\right]},
\end{equation}
to allow for white noise. Unless stated otherwise, we use the latter defintion.
\begin{definition}[{ Convergence}]\label{def: ConDaR_MB}
    A data-driven regularization $\lbrace R_{P_\varepsilon} \rbrace_{ P_\varepsilon \in \Psi}$ is called \emph{convergent} if for any $y \in \mathcal{D}(A^\dagger)\subset Y$ there exists a parameter choice rule $(\delta, P_\varepsilon) \mapsto \Tilde{P}_\varepsilon$ such that 
    \begin{equation*}
        \lim_{\delta \rightarrow 0}\;\sup\left\lbrace \quad\expect_{\varepsilon \sim P_\varepsilon}\left[ \left\| A^\dagger y - R_{\Tilde{P}_\varepsilon}(y+\varepsilon)\right\|^2 \right]\quad \Bigl | \quad P_\varepsilon \in \Psi,\,\boldsymbol{\delta}(P_\varepsilon) \leq \delta  \quad\right\rbrace = 0.
    \end{equation*}
    We call the regularization convergent over the fixed data distribution $P_x$ if there exists a parameter choice rule $(\delta, P_\varepsilon) \mapsto \Tilde{P}_\varepsilon$ such that
    \begin{equation*}
        \lim_{\delta \rightarrow 0}\;\sup\left\lbrace \;\expect_{x \sim P_x, \varepsilon \sim P_\varepsilon}\left[ \left\| A^\dagger Ax - R_{\Tilde{P}_\varepsilon}(Ax+\varepsilon)\right\|^2 \right]\;\Bigl | \; P_\varepsilon \in \Psi,\,\boldsymbol{\delta}(P_\varepsilon) \leq \delta   \;\right\rbrace = 0.
    \end{equation*}
\end{definition}

\subsection{Self-supervised regularization with spectral plug-and-play priors}
In \cite[Section 4.2]{Burger2025Learned_MB} it is pointed out that the reconstruction operator $R^{\text{MSE}}_{P_\varepsilon}$ can also be obtained with a self-supervised learning approach. More precisely, it is considered to optimize a linear denoiser of the form
\begin{equation}\label{mb_eq:denoiser}
    D[\lambda](x) = \sum_{n = 1}^\infty \frac{1}{1+\lambda_n}\langle x, u_n \rangle u_n
\end{equation}
with respect to the self-supervised denoising objective
\begin{equation}\label{mb_eq:denoisingobjective}
    \min_\lambda \expect_{x\sim P_x,\,\tilde\varepsilon \sim P_{\tilde{\varepsilon}}}\left[\left\|D[\lambda](x+\tilde{\varepsilon}) - x\right\|^2\right].
\end{equation}
As before, $x\sim P_x$ is drawn from a data distribution and $\tilde{\varepsilon} \sim P_{\tilde{\varepsilon}}$ is drawn from a noise distribution, but in contrast to \eqref{mb_eq:mseobjective}, the noise is directly added to the data, with no measurement operator in between. 
For a data distribution that fulfills $\Pi_n > n$ for all $n \in \N$ the optimal denoiser \eqref{mb_eq:denoiser} with respect to \eqref{mb_eq:denoisingobjective} is given by 
\begin{equation*}
    D^*_{P_{\tilde{\varepsilon}}}:= D[\lambda^*_{P_{\tilde{\varepsilon}}, P_x}],\qquad \text{with}\; \lambda^*_{P_{\tilde{\varepsilon}}, P_x} = \left(\frac{\expect_{\tilde{\varepsilon}\sim P_{\tilde{\varepsilon}}}\left[\langle \tilde{\varepsilon}, u_n\rangle^2\right]}{\Pi_n}\right)_{n \in \N}.
\end{equation*}
Furthermore, if $\lambda_n \geq 0$ for all $n$, the denoiser of the form \eqref{mb_eq:denoiser} can be written as the proximal operator
\begin{equation*}
    \operatorname{prox}_{J}(x):= \operatorname{arg\,min}_{z \in X} \frac{1}{2} \|x-z\|^2 + J[\lambda](z)
\end{equation*}of the quadratic regularization functional
\begin{equation*}
    J[\lambda](x) = \frac{1}{2} \sum_{n \in \N} \lambda_n \langle x, u_n \rangle^2.
\end{equation*}
Therefore, it can be used as a so-called plug-and-play prior \cite{venkatakrishnan2013pnp_MB} in order to minimize the variational problem
\begin{equation}\label{mb_eq:varprob}
    \frac{1}{2}\|Ax-y^\delta\|^2 + J[\lambda](x)
\end{equation}
that corresponds to the inverse problem \eqref{mb_eq:invproblem}. Since $J[\lambda]$, $D[\lambda]$ and $A$ are all diagonal with respect to the same basis, it turns out that the minimizer of \eqref{mb_eq:varprob} is given by \begin{equation*}
    R\left[\left(\frac{\sigma_n}{\sigma_n^2 + \lambda_n}\right)_{n\in \N}\right](y^\delta),
\end{equation*}
which is equivalent to $R^{\text{MSE}}_{P_\varepsilon}$ if $\lambda_n = \Delta_n/\Pi_n$, which holds true for $D^*_{P_{\tilde{\varepsilon}}}$ if 
\begin{equation}\label{mb_eq:noisenoise}
    \expect_{\tilde{\varepsilon}\sim P_{\tilde{\varepsilon}}}\left[\langle \tilde{\varepsilon}, u_n\rangle^2\right]
    = \Delta_n \qquad \text{for all }n\in\N.
\end{equation}

In other words, if the noise used to train the denoiser behaves in the same way as the noise that appears in the inverse problem, $R^{\text{MSE}}_{P_\varepsilon}y^\delta$ is the minimizer of \eqref{mb_eq:varprob} with $\lambda = \lambda^*_{P_{\tilde{\varepsilon}}, P_x}$. In practice however, the variational problem \eqref{mb_eq:varprob} is not solved directly, but iteratively. In \cite{venkatakrishnan2013pnp_MB} for example, it is proposed to integrate the denoiser into ADMM \cite{glowinski1975approximation_MB, gabay1976admm_MB, Boyd2010admm_MB}.  Hence, the proximal operator required to solve \eqref{mb_eq:varprob} needs to correspond to an appropriate step-size $\tau$ for the chosen iterative method. In general, this means to find $\lambda_\tau$ such that 
\begin{equation}\label{mb_eq:proxscaling}
    D[\lambda_\tau] = \operatorname{prox}_{\tau J[\lambda^*_{P_{\tilde{\varepsilon}}, P_x}]}.
\end{equation}
In our setting, it is straight-forward to compute $\lambda_{\tau} = \tau \lambda^*_{P_{\tilde{\varepsilon}}, P_x} = \tau  \Delta_n / \Pi_n$, which corresponds to scaling the noise distribution by $\tau$. In particular, if $P_{\tilde{\varepsilon}}$ satisfies \eqref{mb_eq:noisenoise} then the denoiser $D^*_{\tau P_{\tilde{\varepsilon}}}$ is the \enquote{correct} choice for an iterative scheme with step-size $\tau$. 
We further note that for general linear denoisers, it was shown in \cite{Hauptmann2024convergent_MB} that the scaling by $\tau$ performed in \eqref{mb_eq:proxscaling} is equivalent to spectral filtering of the denoiser by the filter function $h_{\tau}: \R \rightarrow \R$,
\begin{equation*}
    h_\tau(s) = \frac{s}{\tau - s(\tau-1)}.
\end{equation*}
In our case, this would mean to filter the eigenvalues of $D[\lambda]$ that are given by $1/(1+\lambda_n)$. Inserting these into the filter function, we obtain
\begin{equation*}
h_{\tau}\left(\frac{1}{1+\lambda_n}\right) = \frac{1}{1+\tau\lambda_n},
\end{equation*}
which is exactly the scaling we would expect from the derivations above.
\section{Convergence rates of data-driven spectral regularizations}\label{mb_sec:convrates}
In \cite{Kabri2023Convergent_MB, Burger2025Learned_MB} we showed that $R^{\text{MSE}}_{ P_\varepsilon}$ is a convergent data-driven regularization over fixed training data distributions for mild assumptions on the family $\Psi$. We summarize this result in the following theorem.
\begin{theorem}\label{mb_thm:convmse}
    For any $P_\varepsilon \in \Psi$, let there exist a constant $c > 0$ such that \begin{equation}\label{mb_eq:cont} \Delta_n \geq c\,\sigma_n \Pi_n
    \end{equation} for $n$ large enough and $\Pi_n > 0$ for all $n \in \N$. Then, the family $\left \lbrace R^{\text{MSE}}_{ P_\varepsilon}\right\rbrace_{P_\varepsilon \in \Psi}$ is a convergent data-driven regularization on $\mathcal{D}(A^\dagger)\cap Y$ and over the fixed distribution $P_x$. In particular it holds for any $y \in \mathcal{D}(A^\dagger)\cap Y$ that
    \begin{equation}\label{mb_eq:limsupfehlery}
        \lim_{\delta \rightarrow 0}\,\sup\left\lbrace \,\expect_{\varepsilon \sim P_\varepsilon}\left[ \left\| A^\dagger y - R^{\text{MSE}}_{ P_\varepsilon}(y+\varepsilon)\right\|^2 \right]\,\Bigl | \, P_\varepsilon \in \Psi,\,\boldsymbol{\delta}(P_\varepsilon) \leq \delta \,\right\rbrace = 0,
        \end{equation}
    and that
     \begin{equation}\label{mb_eq:limsupfehler}
        \lim_{\delta \rightarrow 0}\,\sup\left\lbrace \,\expect_{x \sim P_x, \varepsilon \sim P_\varepsilon}\left[ \left\| A^\dagger Ax - R^{\text{MSE}}_{ P_\varepsilon}(Ax+\varepsilon)\right\|^2 \right]\,\Bigl | \, P_\varepsilon \in \Psi,\,\boldsymbol{\delta}(P_\varepsilon) \leq \delta \,\right\rbrace = 0.
        \end{equation}
    \end{theorem}
    \begin{proof}
        Making use of \eqref{mb_eq:cont}, the continuity of $R^{\text{MSE}}_{ P_\varepsilon}$ for any $P_\varepsilon \in \Psi$ follows from  \cite[Lemma 1]{Burger2025Learned_MB}. It further follows from \cite[Theorem 3]{Kabri2023Convergent_MB} and the corresponding proof that 
        \eqref{mb_eq:limsupfehlery} and \eqref{mb_eq:limsupfehler} hold. This corresponds to the parameter-choice rule $(\delta, P_\varepsilon) \mapsto P_\varepsilon$ and is thus sufficient to show the convergence on $\mathcal{D}\cap Y$ and over the distribution $P_x$.
    \end{proof}

    We now want to study the rate of convergence in \eqref{mb_eq:limsupfehler}. Writing out the MSE in \eqref{mb_eq:limsupfehler} {  as} it was done in \cite{Kabri2023Convergent_MB}, we derive that
    \begin{align}\label{eq: Loss_MB}
    \begin{split}
        \| R[g](Ax+\varepsilon)+A^\dagger A x\|^2
        =&\sum_{n}\left((1-\sigma_n g_n)\langle x, u_n\rangle-g_n \langle \varepsilon,v_n\rangle\rangle\right)^2\\
        =&\sum_n  (1-\sigma_n g_n)^2\langle x, u_n\rangle^2 +g_n^2\langle \varepsilon,v_n\rangle^2 \\&-2(1-\sigma_n g_n)g_n\langle x, u_n\rangle\langle \varepsilon,v_n\rangle.   
    \end{split}
    \end{align}
    Inserting $g=g^{\text{MSE}}_{P_{\varepsilon}, P_x}$ and taking the expected value over $P_x$ and $P_\varepsilon$ we obtain
    \begin{align}\label{mb_eq:errordecomp}
        &\expect_{x \sim P_x, \varepsilon \sim P_\varepsilon}\left[ \left\| A^\dagger Ax - R^{\text{MSE}}_{ P_\varepsilon}(Ax+\varepsilon)\right\|^2 \right] \notag\\
        =&\expect_{x\sim P_x} \left[ \left\| A^\dagger Ax - R^{\text{MSE}}_{ P_\varepsilon}(Ax)\right\|^2 \right] + \expect_{\varepsilon \sim P_{\varepsilon}}\left[ \left\| R^{\text{MSE}}_{ P_\varepsilon}(\varepsilon)\right\|^2 \right]\\=& \sum_{n \in \N} \frac{\Pi_n \, \Delta_n}{\Pi_n \sigma_n^2 + \Delta_n},\notag
    \end{align}
    where from the first to the second line, the mixed terms vanish since $\varepsilon$ has zero mean and is not correlated to $x$.
     Given that all coefficients that appear on the right-hand side are non-negative, we directly get the two estimates
    \begin{align}
        \sum_{n \in \N} \frac{\Pi_n \, \Delta_n}{\Pi_n \sigma_n^2 + \Delta_n} \leq \sum_{n \in \N} \Pi_n \label{mb_eq:pibound}
    \end{align}
     and
    \begin{align}
         \sum_{n \in \N} \frac{\Pi_n \, \Delta_n}{\Pi_n \sigma_n^2 + \Delta_n} \leq \sum_{n \in \N} \frac{\Delta_n}{\sigma_n^2}\label{mb_eq:deltabound}.
    \end{align}
    While \eqref{mb_eq:pibound} is finite, it does not depend on $\delta$ and therefore does not vanish except the data distribution is already concentrated on zero. On the other hand although each summand of \eqref{mb_eq:deltabound} can be bounded by $\delta^2/\sigma_n^2$, the case that the whole series is finite corresponds to very well-posed inverse problems and is thus not of major interest. Thus, the convergence proof in \cite{Kabri2023Convergent_MB} instead considers the mixed estimate
    \begin{equation}\label{mb_eq:splitsum}
        \sum_{n \in \N} \frac{\Pi_n \, \Delta_n}{\Pi_n \sigma_n^2 + \Delta_n} \leq 
        \sum_{n = 1}^N \frac{\Delta_n}{\sigma_n^2} + \sum_{n = N+1}^\infty \Pi_n, 
    \end{equation}
    where $N \in \N$ can be chosen freely. To find $N$ that minimizes the upper bound, we first make assumptions on the decay behavior of $\Delta_n/\sigma_n^2$ and $\Pi_n$.
    \begin{assumption}\label{mb_assmp:decayrates}
        We assume that $\Pi_n = O(n^{-a})$ with $a > 1$ and $\Delta_n/\sigma_n^2 = O(\delta^2 n^{b})$ with $b \in \R$.
    \end{assumption}
    Generalizing the above assumption to an entire set of noise distributions $\Psi$, we can show the following convergence rate.
    \begin{theorem}\label{mb_thm:convratedecayrates}
        Let Assumption \ref{mb_assmp:decayrates} hold for constants $a > 1$, $b\in \R$ for the fixed data distribution $P_x$ and uniformly for all noise distributions $P_\varepsilon \in \Psi$, i.e., there exists a constant $C$, such that 
        \begin{align*}\frac{\Delta_n(P_\varepsilon)}{\sigma_n^2} \leq C \boldsymbol{\delta}(P_\varepsilon)^2n^b
        \end{align*}
        for all $P_\varepsilon \in \Psi$. Then it holds that 
        \begin{equation*}
        \sup\left\lbrace \,\expect_{x \sim P_x, \varepsilon \sim P_\varepsilon}\left[ \left\| A^\dagger Ax - R^{\text{MSE}}_{ P_\varepsilon}(Ax+\varepsilon)\right\|^2 \right]\,\Bigl | \, P_\varepsilon \in \Psi,\,\boldsymbol{\delta}(P_\varepsilon) \leq \delta \,\right\rbrace \lesssim \delta^{2\frac{a-1}{a+b}}.
        \end{equation*}
    \end{theorem}
    \begin{proof}
        The proof can be found in Appendix A.
    \end{proof}
     The regularity of $\Pi_n$ and $\Delta_n/\sigma_n$ can be interpreted as the probabilistic version of source conditions known from classical regularization theory. The connection becomes clear with the next Theorem, in which we consider a more general long-time behavior, analogous to the conditions fulfilled in the source sets $\chi_{\mu,\rho}$.
\begin{assumption}\label{mb_assmp:sourcecondition}
        We assume that $\Pi_n = \sigma_n^{4\mu}\,\beta_n$ and $\Delta_n = \delta^2\,\gamma_n$ with $\beta_n, \gamma_n \geq 0$ for all $n \in \N$ such that
        
        \begin{equation}\label{mb_eq:noisedatasum}
            \sum_{n \in \N}\beta_n^{\frac{1}{1+2\mu}}\gamma_n^{\frac{2\mu}{1+2\mu}} \leq c,
        \end{equation}
        with $c < \infty$.
    \end{assumption}
    Assumption \eqref{mb_assmp:sourcecondition} relates to the source sets $\chi_{\mu,\rho}$ in the following way: Again, we assume that $x = (A^*A)^{\mu}w$, but this time, $w$ is a random variable with
    \begin{equation*}
        \expect_{w} \left[\langle w, u_n\rangle^2\right] = \beta_n.
    \end{equation*}
    Then it follows that $\Pi_n  = \sigma_n^4\beta_n$. The condition \eqref{mb_eq:noisedatasum} is in particular a stronger condition than $\expect_{w} \left[\| w\|^2\right] = \sum_{n \in \N} \beta_n \leq \rho^2$. The latter would be a more direct generalization of the classical requirement that $\|w\| \leq \rho$, but since we do not pose any assumptions on the decay of $\gamma_n$, it is not applicable in our setting. We further note that since we consider quadratic errors, the decay rates we state in the following should be compared to the square of the classical bounds, i.e., to $\delta^{2\frac{2\mu}{1+2\mu}}$. 

    \begin{theorem}\label{mb_thm:convratesourcecondition}
        Let Assumption \ref{mb_assmp:sourcecondition} hold for constants $\mu \geq 0$, $c < \infty$ for the fixed data distribution $P_x$ and all noise distributions $P_\varepsilon \in \Psi$. Then it holds that 
        \begin{equation*}
        \sup\left\lbrace \,\expect_{x \sim P_x, \varepsilon \sim P_\varepsilon}\left[ \left\| A^\dagger Ax - R^{\text{MSE}}_{ P_\varepsilon}(Ax+\varepsilon)\right\|^2 \right]\,\Bigl | \, P_\varepsilon \in \Psi,\,\boldsymbol{\delta}(P_\varepsilon) \leq \delta \,\right\rbrace \lesssim \delta^{2\frac{2\mu}{1+2\mu}}.
        \end{equation*}
    \end{theorem}
    \begin{proof}
        The proof can be found in Appendix B.
    \end{proof}
        \begin{remark}
        If Assumption \ref{mb_assmp:decayrates} is fulfilled with $b > -1$, Assumption \ref{mb_assmp:sourcecondition} follows for any $\mu < \frac{a-1}{2(b+1)}$, and the decay rate obtained by Theorem \ref{mb_thm:convratesourcecondition} with $\mu \rightarrow \frac{a-1}{2(b + 1)}$ coincides with the rate obtained by Theorem \ref{mb_thm:convratedecayrates}. The same arguments close the gap for $b = -1$, since in this case we can send $\mu \rightarrow \infty$ and obtain a  rate of order $\delta^2$.
    \end{remark}
    \begin{remark}
        For deterministic data $x \in X$, we can bound the error by
\begin{eqnarray*}
        \expect_{\varepsilon \sim P_\varepsilon}\left[ \left\| A^\dagger Ax - R_{P_\varepsilon}(Ax+\varepsilon)\right\|^2 \right] &=& \sum_{n \in \N} \frac{\Delta^2_n\langle x,u_n\rangle^2 + \Pi^2_n\sigma^2_n \, \Delta_n}{\left(\Pi_n \sigma_n^2 + \Delta_n\right)^2} \\
        = \sum_{n \in \N} \frac{\frac{\langle x,u_n\rangle^2}{\Pi_n}\Delta_n + \sigma_n^2\Pi_n}{\Pi_n \sigma_n^2 + \Delta_n} && \frac{\Pi_n \, \Delta_n}{\Pi_n \sigma_n^2 + \Delta_n} \\ \leq \max\left\{1,\sup_{n\in \N} \frac{\langle x,u_n\rangle^2}{\Pi_n}\right\} && \sum_{n \in \N}\frac{\Pi_n \, \Delta_n}{\Pi_n \sigma_n^2 + \Delta_n}.
    \end{eqnarray*}
Therefore, the convergence rates stated in the theorems above also hold for fixed $x$ with $\langle x,u_n \rangle^2 = O(\Pi_n)$.

Let us mention that from the additional factor $\frac{\frac{\langle x,u_n\rangle^2}{\Pi_n}\Delta_n + \sigma_n^2\Pi_n}{\Pi_n \sigma_n^2 + \Delta_n}$ in the sum above, we see how the relative smoothness of the special solution $x$ (decay of $\langle x, u_n \rangle^2$) and the training data set impacts the convergence. Clearly, if the quotient $\frac{\langle x,u_n\rangle^2}{\Pi_n}$ is tending to infinity, we obtain a slower convergence for the regularization. If $x$  is much smoother than the training data set, we shall obtain an improved convergence rate. Indeed, under the natural condition 
$\frac{\sigma_n^2 \Pi_n}{\Delta_n} \rightarrow 0$, we then obtain 
$$\frac{\frac{\langle x,u_n\rangle^2}{\Pi_n}\Delta_n + \sigma_n^2\Pi_n}{\Pi_n \sigma_n^2 + \Delta_n}
\approx \frac{\langle x,u_n\rangle^2}{\Pi_n} $$
unless $\frac{\langle x,u_n\rangle^2}{\Pi_n} $  decays as fast or even faster than $\frac{\sigma_n^2 \Pi_n}{\Delta_n}$, which is then the correct asymptotic of the full quotient. The latter represents a saturation effect inherent in many regularization methods for ill-posed inverse problems (cf. \cite{engl1996regularization_MB,mathe2004saturation_MB}), which is now depending on the underlying properties of the training data.
\end{remark}
\section{Unsupervised data-driven regularization based on diagonal frame decompositions}
One shortcoming of the spectral parametrization $R$ described before is that it is tied to the singular system of $A$ which might not fit the expected properties of samples from $P_x$ and $P_\varepsilon$. In \cite{Ebner2023Frames_MB, Hubmer2022Frames_MB} it is proposed to instead use a diagonal decomposition based on frames (see e.g., \cite[Chapter 3]{Daubechies1992Wavelets_MB}). In this section we want to generalize our findings from the previous section to this setting. Motivated by the supervised learning objective \eqref{mb_eq:mseobjective}, we derive unsupervised, data-driven regularization based on so-called diagonal frame decompositions. For a more general view on frame-based regularizations, including non-linear filters and learning approaches, we refer the reader to \cite{Ebner2025Frames_MB}.

We start by stating some basic definitions and properties on frames from \cite{Daubechies1992Wavelets_MB}.
Frames can be seen as a generalization of orthonormal bases, where the orthogonality condition is weakened. 
\begin{definition}[Frame]
 A sequence $(\varphi_n)_{n\in\N} \subset Z$ for a linear subspace $Z$ of $X$ is called a \emph{frame} if there exist constants $0 < a \leq b < \infty$ such that for any $x \in Z$ 
\begin{equation}\label{mb_eq:framebounds}
    a\,\|x\|^2 \leq \sum_{n \in \N} \langle x, \varphi_n\rangle^2 \leq b \|x\|^2.
\end{equation}
\end{definition}

Frames can be used to decompose data into coefficients and to reassemble coefficients to data, just like orthonormal bases. Given a frame $(\varphi_n)_{n\in\N}$ we call $F: X \rightarrow \ell^2(\N)$ 
\begin{equation*}
    Fx := (\langle x, \varphi_n\rangle)_{n \in \N}
\end{equation*}
the analysis operator, and its adjoint $F^*: \ell^2(\N) \rightarrow X$, which is given for $c = \{c_n\}_{n\in \N} \in \ell^2(\N)$ by 
\begin{equation*}
    F^*c = \sum_{n\in \N} c_n\, \varphi_n,
\end{equation*}
the synthesis operator. 
In general, it does not hold that $F^*F = \operatorname{Id}$, but the frame bounds described by \eqref{mb_eq:framebounds} ensure that $F^*F$ is invertible on $Z$ (see \cite[Lemma 3.2.2]{Daubechies1992Wavelets_MB}).
Therefore, we can define the so-called dual frame $(\bar{\varphi}_n)_{n\in\N}$
\begin{equation*}
    \bar{\varphi}_n := (F^*F)^{-1}\varphi_n,
\end{equation*}
as well as the corresponding analysis operator $\bar{F}$ and synthesis operator $\bar{F}^*$. It is shown in \cite[Proposition 3.2.3]{Daubechies1992Wavelets_MB} that $\bar{\varphi}$ is indeed a frame (with constants $1/a$ and $1/b$) and that
\begin{equation*}
    \bar{F}^*Fx = \sum_{n \in \N} \langle x, \varphi_n \rangle \bar{\varphi}_n = x,
\end{equation*}
as well as 
\begin{equation*}
    F^*\bar{F}x = \sum_{n \in \N} \langle x, \bar{\varphi}_n \rangle \varphi_n = x,
\end{equation*}
for any $x \in Z$. To minimize confusion caused by the possible non-orthogonality of $\varphi$ in the series expression, we use the notation with the analysis and synthesis operators as often as possible. Thus, we state a practical implication of \eqref{mb_eq:framebounds} for $F^*$ and $\bar{F}^*$ in the following Lemma.
\begin{lemma}\label{mb_lem:F*bound}
    Let $F^*$ be the synthesis operator corresponding to a frame $\varphi$ for $Z \subseteq X$ fulfilling \eqref{mb_eq:framebounds} with constants $0 < a \leq b < \infty$. Further $\bar{F}^*$ be the synthesis operator corresponding to the dual frame $\bar{\varphi}$. Then it holds that
    \begin{equation*}
        \|F^*c\|^2 \leq b\, \|c\|^2_{\ell^2} \qquad \text{and} \qquad \|\bar{F}^*c\|^2 \leq \frac{1}{a}\, \|c\|^2_{\ell^2}
    \end{equation*}
    for any $c \in \ell^2$.
\end{lemma}
\begin{proof}
It follows directly from \eqref{mb_eq:framebounds} and the definition of $F$ that $\|F\|^2 \leq b$.
Since $\|F\| = \|F^*\|$ this proves the inequality for $F^*$. For $\bar{F}^*$ we can use the same arguments since $\bar{\varphi}$ is a frame with upper frame bound constant $1/a$.
\end{proof}
The idea of diagonal frame decomposition is now to write the forward operator $A$ as a diagonal operation on frame coefficients encoding data in $X$ and $Y$, {  similar as} the singular value expansion \eqref{mb_eq:svd}.
\begin{definition}[Diagonal Frame Decomposition, {\cite[Definition 2.4]{Ebner2023Frames_MB}}]\label{mb_def:dfd}
    A system $(\varphi, \psi , \kappa) = (\varphi_n, \psi_n, \kappa_n)_{n\in\N}$ is called a \emph{diagonal frame decomposition} of $A$ if 
    \begin{enumerate}[label=(\arabic*)]
        \item $\varphi = (\varphi_n)_{n \in \N}$ is a frame for $\mathcal{N}(A)^\perp \subseteq X$,
        \item $\psi = (\psi_n)_{n \in \N}$ is a frame for $\overline{\mathcal{R}(A)} \subseteq Y$,
        \item and the quasi-singular values $\kappa = (\kappa_n)_{n\in\N} \subset \R_{>0}$ fulfill
        \begin{equation*}
            A^*\psi_n = \kappa_n \varphi_n  \qquad \text{for all } n\in \N.
        \end{equation*}
    \end{enumerate}
\end{definition}
\begin{remark}\label{mb_rem:quasisingeq}
    The last condition in Definition \ref{mb_def:dfd} can be rewritten in terms of $F_\varphi$ and $F_\psi$, i.e., the analysis operators corresponding to $\varphi$ and $\psi$, as 
    \begin{equation*}
    F_{\psi}(Ax) = \kappa \cdot F_{\varphi}(x),
    \end{equation*}
    for any $x \in X$.
\end{remark}
It is shown in \cite[Theorem 2.9]{Ebner2023Frames_MB} that a diagonal frame decomposition $(\varphi, \psi, \kappa)$ of $A$ decomposes the pseudo-inverse $A^\dagger$ analogously to \eqref{mb_eq:pseudoinv},
\begin{equation*}
    A^\dagger y = \sum_{n \in \N} \frac{1}{\kappa_n}\, \langle y, \psi_n\rangle \,\bar{\varphi}_n = \bar{F}_\varphi^*(1/\kappa\cdot F_\psi y^\delta)
\end{equation*}
for any $y \in \mathcal{D}(A^\dagger)$. 
The essential generalization with respect to \eqref{mb_eq:pseudoinv} is that neither $\varphi$ nor $\psi$ are required to be orthonormal, and therefore, the coefficients have to be reassembled with the synthesis operator corresponding to the \textit{dual} frame $\bar{\varphi}$. Together with Remark \ref{mb_rem:quasisingeq} this decomposition of $A^\dagger$ further yields that
\begin{equation}\label{mb_eq:projnperp}
    A^\dagger A  = \bar{F}^*_{\varphi}F_{\varphi}    
\end{equation}
for any $x \in X$, which means that $\bar{F}^*_{\varphi}F_{\varphi}$ performs the orthogonal projection onto $\mathcal{N}(A)^\perp$. It is now natural to consider the frame generalization of the spectral reconstruction operators \eqref{mb_eq:spectralreg},
\begin{equation}\label{mb_eq:framerec}
    R[g](y^\delta) = \sum_{n \in \N} g_n\, \langle y^\delta, \psi_n\rangle\,\bar{\varphi}_n = \bar{F}_\varphi^*(g\cdot F_\psi y^\delta),
\end{equation}
where we omit the dependence of $R$ on $\varphi$ and $\psi$ for the sake of readability.
With such reconstruction operators we can again consider the supervised learning problem which aims to solve \eqref{mb_eq:mseobjective}.
Since $R$ still maps to $\mathcal{N}(A)^\perp$, solving \eqref{mb_eq:mseobjective} is equivalent to minimizing
\begin{align*}
    &\expect_{x \sim P_x, \varepsilon \sim P_\varepsilon}\left[ \left\| A^\dagger Ax - R[g](Ax + \varepsilon)\right\|^2 \right]\\
    &=\expect_{x \sim P_x}\left[ \left\| \bar{F}^*_{\varphi}F_{\varphi} x - \bar{F}_\varphi^*(\kappa \cdot g\cdot F_\varphi (x))\right\|^2 \right] + \expect_{\varepsilon \sim P_\varepsilon}\left[ \left\| \bar{F}_\varphi^*(g\cdot F_\psi \varepsilon)\right\|^2 \right],
\end{align*}
where, as in \eqref{mb_eq:errordecomp}, we use that $\varepsilon$ has zero mean and is not correlated to $x$ to decompose the error, and the identities derived in \eqref{mb_eq:projnperp} and Remark \ref{mb_rem:quasisingeq} to rewrite the operators. The above form allows the direct application of Lemma \ref{mb_lem:F*bound} that yields
\begin{align}\label{mb_eq:upperboundframeerror}
    &a_\varphi\,\expect_{x \sim P_x, \varepsilon \sim P_\varepsilon}\left[ \left\| A^\dagger Ax - R[g](Ax + \varepsilon)\right\|^2 \right]\notag \\&\leq  \expect_{x \sim P_x}\left[ \left\|(1-\kappa \cdot g)\,F_{\varphi} x\right\|_{\ell^2}^2 \right]\; + \;\expect_{\varepsilon \sim P_\varepsilon}\left[ \left\|g\cdot F_\psi \varepsilon\right\|_{\ell^2}^2 \right]\\
    &= \sum_{n \in \N} (1-\kappa_ng_n)^2\, \expect_{x \sim P_x}\left[ \langle x, \varphi_n\rangle^2\right] + g_n^2\,\expect_{\varepsilon \sim P_\varepsilon}\left[ \langle \varepsilon, \psi_n\rangle^2\right],\notag
\end{align}
where $a_\varphi$ denotes the constant for the lower bound of $\varphi$.

\noindent Assuming that $\expect_{x \sim P_x}\left[ \langle x, \varphi_n\rangle^2\right] > 0$ for any $n > 0$ and $\expect_{x \sim P_x}\left[ \| x\|^2\right] < \infty$, this upper bound has a unique minimizer $g^*$ given by 
\begin{equation*}
    g^* :=
    \left(\frac{\kappa_n\, \expect_{x \sim P_x}\left[ \langle x, \varphi_n\rangle^2\right]}{\kappa_n^2 \expect_{x \sim P_x}\left[ \langle x, \varphi_n\rangle^2\right] + \expect_{\varepsilon \sim P_\varepsilon}\left[ \langle \varepsilon, \psi_n\rangle^2\right]}\right)_{n \in \N}.
\end{equation*}
This turns out to be the generalization of $g^{\text{MSE}}_{P_{\varepsilon},P_x}$ in \eqref{mb_eq:gmse} if we expand the definition of $\Pi_n$, $\Delta_n$ to 
\begin{equation}\label{mb_eq:generalizedDeltaPi}
    \Pi_n = \expect_{x \sim P_x}\left[\left|\langle x, \varphi_n\rangle\right|^2\right] \qquad \text{and} \qquad \Delta_n = \expect_{x \sim P_{\varepsilon}}\left[\left|\langle \varepsilon, \psi_n\rangle\right|^2\right]
\end{equation}
for any choice of frames $\varphi$, $\psi$ that belong to a diagonal frame decomposition $(\varphi, \psi, \kappa)$ of $A$. We note that using this generalized definition, we still obtain summability of the $\Pi_n$, as
    \begin{equation*}
    \sum_{n \in \N} \Pi_n = \expect_{x\sim P_x}\left[\sum_{n \in \N}\langle x,\varphi_n\rangle^2\right] \leq b_\varphi\, \expect_{x\sim P_x}\left[\|x\|^2\right] < \infty,
\end{equation*}
where we use the upper frame bound \eqref{mb_eq:framebounds} for the frame $\varphi$ with constant $b_\varphi < \infty$.
Further, inserting \eqref{mb_eq:generalizedDeltaPi} into \eqref{mb_eq:upperboundframeerror}, we derive that
\begin{align*}
        & a_\varphi\,\expect_{x \sim P_x, \varepsilon \sim P_\varepsilon}\left[ \left\| A^\dagger Ax - R^{\text{MSE}^*}_{ P_\varepsilon}(Ax+\varepsilon)\right\|^2 \right] \\&\leq \sum_{n \in \N} (1-\kappa_n(g^{\text{MSE}^*}_{P_\varepsilon,P_x}))^2\, \Pi_n + (g^{\text{MSE}^*}_{P_\varepsilon,P_x})_n^2\,\Delta_n = \sum_{n\in \N} \frac{\Pi_n \Delta_n}{\kappa_n^2\Pi_n + \Delta_n}.
    \end{align*}
Combining this estimate with the summability of $\Pi_n$, we can generalize most of the results from Section \ref{mb_sec:convrates} to the case of frame-based reconstruction operators. Therefore, we transfer the noise level defined by \eqref{mb_eq:noiselevelDelta} to a given diagonal frame decomposition $(\varphi, \psi, \kappa)$ as
\begin{equation*}
    \boldsymbol{\delta}(P_\varepsilon) = \sqrt{\sup_{n \in \N}\expect_{\varepsilon \sim P_\varepsilon}\left[\left|\langle\varepsilon, \psi_n\rangle\right|^2\right]}.
\end{equation*}
\begin{definition}
    Let $(\varphi, \psi, \kappa)$ be a diagonal frame decomposition of $A$ and consider the corresponding generalized definitions \eqref{mb_eq:framerec} of $R$ and \eqref{mb_eq:generalizedDeltaPi} of $\Pi_n$, $\Delta_n$. Then we define
    \begin{equation*}
        R^{\text{MSE}^*}_{P_\varepsilon} := R[g^{\text{MSE}^*}_{P_\varepsilon,P_x}],  
    \end{equation*}
    {  where}
    \begin{equation*}
    g^{\text{MSE}^*}_{P_\varepsilon,P_x} :=  \left(\frac{\kappa_n\, \Pi_n}{\kappa_n^2 \Pi_n + \Delta_n} \right)_{n \in \N}.
\end{equation*}
\end{definition}
\begin{theorem}
    For any $P_\varepsilon \in \Psi$, let there exist a constant $c > 0$ such that 
    \begin{equation}\label{mb_eq:cont_frames} \Delta_n \geq c\,\kappa_n \Pi_n
    \end{equation} 
    for $n$ large enough. Then, the family $\left \lbrace R^{\text{MSE}^*}_{ P_\varepsilon}\right\rbrace_{P_\varepsilon \in \Psi}$ is a convergent data-driven regularization over the fixed distribution $P_x$. In particular it holds that
     \begin{equation*}
        \lim_{\delta \rightarrow 0}\,\sup\left\lbrace \,\expect_{x \sim P_x, \varepsilon \sim P_\varepsilon}\left[ \left\| A^\dagger Ax - R^{\text{MSE}^*}_{ P_\varepsilon}(Ax+\varepsilon)\right\|^2 \right]\,\Bigl | \, P_\varepsilon \in \Psi,\,\boldsymbol{\delta}(P_\varepsilon) \leq \delta \,\right\rbrace = 0.
        \end{equation*}
        If there exist constants $a > 1$ and $b \in \R$, such that $P_x$ fulfills $\Pi_n = O(n^{-a})$ and $\Delta_n/\kappa_n^2 = O(\delta^2 n^b)$ holds uniformly for all $P_\varepsilon \in \Psi$, the speed of convergence is of order $\delta^{2\frac{a-1}{a+b}}$. 
        If there exist constants $\mu \geq 0$ and $c < \infty$, such that $P_x$ fulfills $\Pi_n = \kappa_n^{4\mu}\beta_n$ and for all $P_\varepsilon \in \Psi$ there exist $\gamma_n$ such that $\Delta_n = \delta^2\gamma_n$ and \eqref{mb_eq:noisedatasum} holds, the speed of convergence is of order $\delta^{2\frac{2\mu}{1+2\mu}}$.\end{theorem}
\begin{proof}
The proof of convergence follows analogously to Theorem \ref{mb_thm:convmse}. The decay rates follow directly from Lemmas \ref{mb_lem:convratedoubleexp} and \ref{mb_lem:decaysourcecondition} by substituting the singular values $\sigma_n$ by the quasi-singular values $\kappa_n$.
\end{proof}
 At this point we would like to stress that, although it is clearly data-driven by the coefficients $\Pi_n$ and $\Delta_n$, the regularizer $R^{\text{MSE}^*}_{P_\varepsilon}$ is in general not obtained by the supervised learning approach. In fact, the only case where we know that $R^{\text{MSE}^*}_{P_\varepsilon}$ minimizes the learning objective \eqref{mb_eq:mseobjective} is when the inequality \eqref{mb_eq:upperboundframeerror} becomes an equation. 

\section{Regularization aspects of adversarial training}
To motivate the use of adversarial training as a regularization let us consider the following setting.   In the context of Chapter \ref{ch:_ClassRegMB} we consider a yet undetermined spectral regularizer $R_\alpha := R[ g(\alpha)]$ \eqref{mb_eq:spectralreg} with coefficients $g_n$ depending on the parameter $\alpha\in \R_{>0}$.
Let us further fix the parameter choice rule to simply be  $\alpha(\delta,y^\delta)=\delta$. 
Since $R_\alpha := R[ g(\alpha)]$ the Definition \ref{def: Reg_MB} is equivalent to
\begin{gather*}
    \lim_{\delta \rightarrow 0} \sup_{\varepsilon\in\overline{B}_\delta(0)} \| R_\delta(Ax +\varepsilon)- A^\dagger Ax\|^2 =0 
\end{gather*}
or 
\begin{gather*}
    \lim_{\delta \rightarrow 0} \sup_{\varepsilon\in\overline{B}_\delta(0)} \| R_\delta(Ax +\varepsilon)- x\|^2-\|x_0 \|^2 =0,
\end{gather*}
where $x_0\in \mathcal{N}(A)$  is the unique projection of $x$ onto the null space of the operator $A$.
\begin{remark}
Choosing the term  $\| R[g](Ax+\varepsilon)+x\|^2$ in the following minimization objective is more in line with the classical machine learning approach but we have the equivalence
    \begin{align*}
        \| R[g](Ax+\varepsilon)+x\|^2=\| R[g](Ax+\varepsilon)+A^\dagger A x\|^2 +\|x_0\|^2
    \end{align*}
where $x_0\in \mathcal{N}(A)$  is the unique projection of $x$ onto the null space of the operator.
Since $\|x_0\|^2$ is a $g$ independent constant we can always switch between the two terms in the minimization objectives.
\end{remark}
Given a data distribution $P_x$ if we aim to have a minimal regularization error over all $x$ it seems natural to choose the parameters $g_n(\delta)$ such that
\begin{align}\label{eq: AdTr_MB}
g_n^{adv_2}\in\arg\min_{g_n} \mathbb{E}_{x\sim P_x}\sup_{{\|\varepsilon\|\leq \delta} }\| R_{D}(Ax+\varepsilon)+x\|^2.
\end{align}
In machine learning this training objective is also called adversarial training (cf. \cite{goodfellow2014explaining_MB,peck2023introduction_MB, weigand2024adversarial_MB} ).
A priori it is unknown if the infimum in the above problem is attained by a minimizer. But as we will see shortly this is indeed the case. 
\begin{lemma} \label{lm: Con_MB}The term
\begin{align*}
    \sup_{\|\varepsilon\|\leq \delta} \| R[g](Ax+\varepsilon)+x\|^2
\end{align*}
 is finite if and only if $R[g]$ is a bounded linear operator, i.e. $g=(g_n)_{n\in\N}\in \ell^\infty$. In this case it can be bounded by $2\|R[g](Ax)-x\|^2+2\|g\|_{\ell^\infty}^2\delta^2 $.   
\end{lemma}
\begin{proof}
If $g\in \ell^\infty$, then by \eqref{eq: Loss_MB}, Cauch--Schwartz and Young's inequality we can estimate 
\begin{align*}
     \| R[g](Ax+\varepsilon)+x\|^2&\leq\|x_0\|^2+2\sum_n  (1-\sigma_n g_n)^2\langle x, u_n\rangle^2 +g_n^2\langle \varepsilon,v_n\rangle^2\\
     &\leq  \|x_0\|^2+2\sum_n  (1-\sigma_n g_n)^2\langle x, u_n\rangle^2 + \|g_n\|_{\ell^\infty}^2 \delta^2.
\end{align*}
On the other hand if $g\not\in l^\infty$ then $g_n \rightarrow \infty$ and thus choosing $\varepsilon_n = \sign(-2(1-\sigma_n g_n) g_n \langle x,u_n\rangle) \delta v_n$ the sequence $\|R[g](Ax+\varepsilon_n)+x\|^2 \rightarrow +\infty$.
\end{proof}
Due to Lemma \ref{lm: Con_MB} minimizers of \eqref{eq: AdTr_MB} have to lie in $\ell^\infty$ and thus we can restrict ourselves to the case $g\in \ell^\infty$.
Since unlike in the MSE case we have no closed form for the minimizer $g^{adv_2}$. We have to rely on the the direct method of calculus of variations to show the existence of minimizers. The two essential ingredients for this method are the weak* lower semicontinuity and the compactness of sublevels for the functional
\begin{align*}
    \bar{f}: g \mapsto \mathbb{E}_{x\sim P_x} \sup_{\|\varepsilon\|\leq \delta}  \|R[g](Ax+\varepsilon) -A^\dagger Ax\|^2.
\end{align*}
The proofs of the following two lemmas can be found in Appendix C.
\begin{lemma}\label{lm:fR_MB}
The function 
    \begin{align*}
        f:(g,x,\varepsilon)\in \ell^\infty\times X \times Y \mapsto \|R[g](Ax+\varepsilon) -A^\dagger Ax\|^2
    \end{align*}
is weakly * lower semi continuous in $g$ and jointly continuous in $x$ and $\varepsilon$. Further for a  $\delta>0$ the functional
\begin{align*}
    \tilde{f}: (g,x) \in \ell^\infty \times X \mapsto \sup_{\|\varepsilon\|\leq \delta}  \|R[g](Ax+\varepsilon) -A^\dagger Ax\|^2
\end{align*}
is Borel measurable in $x$ and weakly* lower semi continuous in $g$. Lastly the functional 
\begin{align*}
    \bar{f}: g\in \ell^\infty  \mapsto \mathbb{E}_{x\sim P_x} \sup_{\|\varepsilon\|\leq \delta}  \|R[g](Ax+\varepsilon) -A^\dagger Ax\|^2
\end{align*}
is weak* lower semi continuous.
\end{lemma}

\begin{lemma}(Compactness of Sublevels)\label{lm:SL_MB}
    Given a fixed $\delta>0$ we have that the functional $\bar{f}$ has weakly* compact sublevels.
\end{lemma}
As an immediate consequence of Lemma \ref{lm:fR_MB} and \ref{lm:SL_MB} we can apply the direct method in the calculus of variations and obtain the following corollary.
\begin{cor}
    The minimum in \eqref{eq: AdTr_MB} is attained.
\end{cor}

We define the regularization procedure given by adversarial training $R^{adv_2}_{\delta}\coloneqq R[g^{adv_2}]$ where  $g^{adv_2}$ are the parameters obtained for the problem \eqref{eq: AdTr_MB} with adversarial budget $\delta$.
By Lemma \ref{lm: Con_MB} $R^{adv_2}_\delta$ is a continuous operator.  To show that $R^{adv_2}_{\delta}$ is a valid regularization scheme we need to study its behavior as $\delta \rightarrow 0$. Indeed adversarial training of spectral regularizers yields a convergent regularization scheme in the following sense.
\begin{lemma}\label{lm: A2C_MB}
If the training data distribution fulfills $\mathbb{E}_{x\sim P_x} \left[ \|x\|^2\right] <+\infty$, then
 $$ \lim_{\delta \rightarrow 0}\mathbb{E}_{x\sim P_x} \sup_{\|\varepsilon\|<\delta}\| R_{\delta}^{adv_2} (Ax+\varepsilon)-A^\dagger A x\| = 0 .
 $$ 
\end{lemma}
\begin{proof}
 We estimate  
 \begin{align*}
     \sum_n  (1-\sigma_n g_n)^2\langle x, u_n\rangle^2 +g_n^2\langle \varepsilon,v_n\rangle^2-2(1-\sigma_n g_n)g_n\langle x, u_n\rangle\langle \varepsilon,v_n\rangle\\
     \leq 2\sum_n  (1-\sigma_n g_n)^2\langle x, u_n\rangle^2 +g_n^2\langle \varepsilon,v_n\rangle^2.
 \end{align*}
 Thus
 \begin{align*}
&\min_{g_n} E_{x\sim P_x}\sup_{\|\varepsilon\|\leq\delta}  \|R_{g_n}(Ax+\varepsilon)-A^\dagger A x\|^2\\
& \qquad\leq   2\ \min_{g_n}  E_{x\sim P_x}\sum_n  (1-\sigma_n g_n)^2\langle x, u_n\rangle^2 +\sup_{\|\varepsilon\|\leq\delta}   \sum_n g_n^2\langle \varepsilon,v_n\rangle^2\\
& \qquad=2\ \min_{g_n}  \sum_n  (1-\sigma_n g_n)^2 \Pi_n + \delta^2 \|g\|_{\ell^\infty}^2\\
& \qquad= 2\min_{ r>0} \sum_n (1-\sigma_n \min(r,1/\sigma_n)) \Pi_n+\delta^2 r^2
\\
& \qquad\leq 2 \sum_n (1-\sigma_n\min(1/\sqrt{\delta},1/\sigma_n))\Pi_n+\sqrt{\delta},
 \end{align*}
 where in the last step we simply inserted  $r=\frac{1}{\sqrt{\delta}}$.
 The first term converges to $0$ since $\Pi_n \in l^1$ and $\sigma_n \rightarrow 0$ as $n\rightarrow \infty$ and for the second term the convergence is obvious.
\end{proof}

\begin{remark}
If we define the noise level $\delta(P_\varepsilon)\coloneqq \sqrt{ \mathbb{E}_{\varepsilon\sim P_\varepsilon} \|x\|^2}$, then 
\begin{align*}
    \mathbb{E}_{x\sim P_x, \varepsilon\sim P_\varepsilon} \sup_{P_\varepsilon: \delta(P_\varepsilon)  \leq \delta} \|R^{adv_2}_\delta(A x+\varepsilon) -A^\dagger A x\|^2 \\
    \leq \mathbb{E}_{x\sim P_x}\sup_{\varepsilon: \|\varepsilon\|\leq \delta} \|R^{adv_2}_\delta(A x+\varepsilon) -A^\dagger A x\|^2.
\end{align*}
This can easily be observed if we use \eqref{eq: Loss_MB} . The mixed term vanishes on the left hand side since $x$ and $\varepsilon$ are considered to be independent and $x$ has zero mean. For right hand side the mixed term can be assumed to only make positive contributions due to supremum over $\varepsilon$.
Thus by Lemma \ref{lm: A2C_MB}  we obtain
\begin{equation*}
        \lim_{\delta \rightarrow 0}\;\sup\left\lbrace \;\expect_{x \sim P_x, \varepsilon \sim P_\varepsilon}\left[ \left\| A^\dagger Ax - R_{\delta}^{adv_2}(Ax+\varepsilon)\right\|^2 \right]\;\Bigl | \; \boldsymbol{\delta}(P_\varepsilon) \leq \delta   \;\right\rbrace = 0.
    \end{equation*}
which corresponds to convergence in the sense of Definition \ref{def: ConDaR_MB} .

\end{remark}
 \subsection{$\infty$-Type adversaries}
Unfortunately it seems that solutions to the optimization problem \eqref{eq: AdTr_MB} have no closed form like for the MSE-problem. To maintain a closed form we slightly modified the problem. Instead of of the $\delta$-ball we take the supremum with respect to the set
\begin{align*}
S_\delta\coloneqq\left\{ \varepsilon\in X : |\langle \varepsilon,v_n\rangle|\leq \delta \quad \forall n\right\}
\end{align*}
Although this set is not a $\delta$-ball in the infinite dimensional case in finite dimension this set is the $\delta$-ball of a rotated $\infty$-norm. The set $S_\delta$ can be interpreted as the vicinity corresponding to a deterministic version of a white noise level $\delta$.
\begin{lemma}
The term
$$\sup_{\varepsilon\in S_\delta} \|R[g](Ax+\varepsilon)-x\|^2$$ is finite if and only if $g \in \ell^2$.
\end{lemma}
\begin{proof}
    If $g=(g_n)_{n\in\N} \not\in \ell^2$, take the sequence $\Tilde{\varepsilon}_k=\sum_{n=1}^k \delta \sign(-2(1-\sigma_n g_n) g_n\langle x,u_n\rangle)v_n$.
    and estimate 
    \begin{align*}
        \|R[g](Ax+\varepsilon_k)-x\|^2=&\|x_0\|^2+\sum_n  (1-\sigma_n g_n)^2\langle x, u_n\rangle^2 +g_n^2\langle \varepsilon_k,v_n\rangle^2 \\&-2(1-\sigma_n g_n)g_n\langle x, u_n\rangle\langle \varepsilon_k,v_n\rangle\\
        \geq& \sum_n g_n^2 \langle \varepsilon_k,v_n \rangle^2\rightarrow +\infty 
    \end{align*}
    for $k\rightarrow +\infty$. On the other hand if $(g_n)_{n\in\N} \in \ell^2$ then we can look at the functional $f: X\times \ell^\infty \rightarrow \R$ mapping
    \begin{align*}
        (x,w) \mapsto \sum_n  (1-\sigma_n g_n)^2\langle x, u_n\rangle^2-2(1-\sigma_n g_n) g_n \langle x, u_n\rangle w_n+g_n^2 w_n^2. 
    \end{align*}
    Finiteness of the functional can be guaranteed by a combination of Cauchy-Schwartz and Young's inequality. For a fixed $x\in X$ the maximum of this functional under the restriction $\|w\|_{\ell^\infty}\leq \delta$ is clearly obtained by $w_n^{max}=\sign(-2(1-\sigma_n g_n)g_n\langle x,u_n\rangle) \delta$. 
    Then 
    \begin{align*}
        \|R[g](Ax+\Tilde{\varepsilon}_k)-x\|^2 \rightarrow f(x,w^{max})+\|x_0\|^2
    \end{align*}
    for $k\rightarrow +\infty$. In particular
    \begin{align*}
        \sup_{S_\delta} \| R[g](Ax+\varepsilon)-x\|^2 =\max_{w\in \ell^\infty: \|w\|_{\ell^\infty} \leq \delta} f(x,w)+\|x_0\|^2.
    \end{align*}
\end{proof}
Due to the previous considerations we can solve
\begin{align}\label{eq: HAdRed_MB}
    \min_{g\in \ell^2} \mathbb{E}_{x\sim P_x} \max_{w\in \ell^\infty: \|w\|_{\ell^\infty} \leq \delta} f_g(x,w)+\|x_0\|^2=\min_{g\in \ell^2} \mathbb{E}_{x\sim P_x}f_g(x,w^{max}_x)+\|x_0\|^2.
\end{align}
instead of 
\begin{align}\label{eq: AdInf_MB}
    \min_{g} \mathbb{E}_{x\sim P_x}\sup_{S_\delta} \|R[g](Ax+\varepsilon)-x\|^2.
\end{align}
The subscripts $g$ and $x$ denote the dependencies of $f$ and $w^{max}$ on $g$ and $x$, respectively.
\begin{lemma}\label{lm: GRange_MB}
    A minimizer $m$ of problem \eqref{eq: HAdRed_MB} fulfills $m_n\in \left[0,\frac{1}{\sigma_n}\right]$.
\end{lemma}
\begin{proof}
We start by observing that $\mathbb{E}_{x\sim P_x}\langle u, x_n\rangle^2=0$ if and only if 
$\mathbb{E}_{x\sim P_x}|\langle u, x_n\rangle|=0$. Using representation \eqref{eq: Loss_MB}, the $n$-th summand of $\mathbb{E}_{x\sim P_x} f_{g}(x,w^{max}_x)$ is given in this case by $g_n^2\delta^2$ and $m_n$ has to be $0$.\\ 
Let us now consider the case $\mathbb{E}_{x\sim P_x}\langle u, x_n\rangle^2\not=0\not=\mathbb{E}_{x\sim P_x}|\langle u, x_n\rangle|$.
If $m_n<0$ for some $n$, we can set $\Tilde{m}_n=-m_n$
to obtain the contradiction $$ \mathbb{E}_{x\sim P_x}f_{m}(x,w^{max}_x)+\|x_0\|^2> \mathbb{E}_{x\sim P_x} f_{\Tilde{m}}(x,w^{max}_x)+\|x_0\|^2.$$
To see this compare the $n$-th summand of $\mathbb{E}_{x\sim P_x}f_{m}(x,w^{max}_x)$ and $\mathbb{E}_{x\sim P_x} f_{\Tilde{m}}(x,w^{max}_x)$. 
Since $\sigma_n>0$,  $| 1-\sigma_n m_n|> | 1-\sigma_n \Tilde{m}_n|$, clearly 
\begin{align*}
    &(1-\sigma_n m_n)^2 \mathbb{E}_{x\sim P_x}\langle x, u_n\rangle^2 +2\delta |(1-\sigma_n m_n)| |m_n| \mathbb{E}_{x\sim P_x} | \langle u,x_n \rangle|  + (m_n)^2 \delta^2 \\
    &>(1-\sigma_n \Tilde{m}_n)^2 \mathbb{E}_{x\sim P_x}\langle x, u_n\rangle^2 +2\delta |(1-\sigma_n \Tilde{m}_n)| |\Tilde{m}_n| \mathbb{E}_{x\sim P_x}  |\langle u,x_n \rangle|  + (\Tilde{m}_n)^2 \delta^2.
\end{align*}
On the other hand if $m_n>\frac{1}{\sigma_n}$ we can set $\Tilde{m}_n=\frac{1}{\sigma_n}$, then again observe
\begin{align*}
    &(1-\sigma_n m_n)^2 \mathbb{E}_{x\sim P_x}\langle x, u_n\rangle^2
    +2\delta |(1-\sigma_n m_n)| |m_n| \mathbb{E}_{x\sim P_x} | \langle u,x_n \rangle| + (m_n)^2 \delta^2\\
    &> (\Tilde{m}_n)^2 \delta^2\\
    &=\underbrace{(1-\sigma_n \Tilde{m}_n)^2}_{=0} \mathbb{E}_{x\sim P_x}\langle x, u_n\rangle^2
     +2\delta \underbrace{|(1-\sigma_n \Tilde{m}_n)|}_{=0} |\Tilde{m}_n| \mathbb{E}_{x\sim P_x}  |\langle u,x_n \rangle| + (\Tilde{m}_n)^2 \delta^2
\end{align*}
and $\Tilde{m}_n$ is a strict improvement over $m_n$.
\end{proof}

With those insights we are now able to calculate the minimizers precisely. The proof of the following theorem can be found in Appendix D.

\begin{theorem}\label{thm: AdInftyVal}
A minimizer of problem \eqref{eq: HAdRed_MB} is given by $g^{adv_\infty}\in \ell^2 $ with the choice 
\begin{align*}
    g_n^{adv_\infty}=0 \quad\text{if }\  \frac{\mathbb{E}_{x\sim P_x} \langle x, u_n \rangle^2}{\mathbb{E}_{x\sim P_x} |\langle x, u_n \rangle|}\leq\frac{\delta}{\sigma_n}\ \text{ or }\ \mathbb{E}_{x\sim P_x} |\langle u,x_n\rangle|=0,
\end{align*}
and else
\begin{align*}
    g_n^{adv_\infty}=\begin{cases} 
    \frac{1}{\sigma_n-\delta \frac{\sigma_n \mathbb{E}_{x\sim P_x} |\langle x,u_n\rangle|-\delta}{\sigma_n \mathbb{E}_{x\sim P_x} \langle x,u_n\rangle^2-\delta \mathbb{E}_{x\sim P_x} |\langle x,u_n\rangle|}}& \text{if } \mathbb{E}_{x\sim P_x} |\langle x,u_n\rangle|< \frac{\delta}{\sigma_n}< \frac{\mathbb{E}_{x\sim P_x} \langle x, u_n \rangle^2}{\mathbb{E}_{x\sim P_x} |\langle x, u_n \rangle|},\\
    \frac{1}{\sigma_n} &\text{if } \frac{\delta}{\sigma_n}\leq\mathbb{E}_{x\sim P_x} |\langle x,u_n\rangle|.
    \end{cases}
\end{align*}
\end{theorem}

\begin{remark}
In in the case that $\left(\mathbb{E}_{x\sim P_x} |\langle x, u_n\rangle|\right)^2=\mathbb{E}_{x\sim P_x} \langle x, u_n\rangle^2$ and $\frac{\delta}{\sigma_n}=\mathbb{E}_{x\sim P_x} |\langle x, u_n\rangle|$ the term $S(g_n)=\mathbb{E}_{x\sim P_x} \langle x,u_n\rangle^2$ is constant and every choice of $g_n\in\left[0,\frac{1}{\sigma_n}\right]$ is a minimizer. We simply set $g_n=0$ in the definition of $g_n^{adv_\infty}$. In any other case the choice for $g_n^{adv_\infty}$ is unique.
\end{remark}
We denote the regularization obtain by the solving problem \eqref{eq: AdInf_MB} with an $\varepsilon$ budget of $\delta$ by $R^{adv_\infty}_\delta\coloneqq R[g^{adv_\infty}]$. To study the convergence of this regularizer we introduce the set of distributions with a white noise level smaller then $\delta$
\begin{align*}
    \mathcal{S}_\delta \coloneqq \left\{P_\varepsilon: \mathbb{E}_{\varepsilon \sim P_\varepsilon} \langle \varepsilon, v_n\rangle^2\leq \delta^2, \forall n \right\}.
\end{align*}
\begin{theorem}
Let $P_x$ be the training data distribution satisfying $ \mathbb{E}_{x\sim P_x}\left[\|x\|^2\right]<+\infty$.
Then for any fixed data distribution $\Tilde{P}_x$ with $\mathbb{E}_{ x\sim \Tilde{P}_x} \left[ \|x\|^2\right] <+\infty$ 
    we have
\begin{align*}
    \lim_{\delta \rightarrow 0} \mathbb{E}_{x\sim\Tilde{P}_x}\sup_{\varepsilon\in S_\delta} \| R^{adv_\infty}_\delta(Ax +\varepsilon)-A^\dagger Ax\|^2=0 ,
\end{align*}
as well as
\begin{align*}
    \lim_{\delta \rightarrow 0}  \sup_{\Tilde{P}_\varepsilon \in \mathcal{S}_\delta}\mathbb{E}_{x\sim\tilde{P}_x,\varepsilon\sim P_\varepsilon}\| R^{adv_\infty}_\delta(Ax +\varepsilon)-A^\dagger Ax\|^2=0.
\end{align*}
\end{theorem}
\begin{proof}
    For $r>0$ we define $g(r)$ as the sequence 
    \begin{align*}
        g(r)_i\coloneqq \begin{cases}
            \frac{1}{\sigma_i} &\text{if }\sum^i_{1=n} \frac{1}{\sigma_n^2} \leq r^2 \\
            0 &\text{ else}
        \end{cases}
    \end{align*}
    such that by definition $\|g(r)\|\leq r$.
    Let now $P_x$ be the trainings distribution, then
    \begin{align*}
    &\sum_n (1-\sigma_n g_n^{adv_\infty})^2\mathbb{E}_{x\sim P_x} \langle x, u_n\rangle^2 \\
    &+ 2 \delta (1-\sigma_n g_n^{adv_\infty})g_n^{adv_\infty}\mathbb{E}_{x\sim P_x} |\langle x, u_n\rangle|+\delta^2 (g_n^{adv_\infty})^2\\
        =&\min_{g} \mathbb{E}_{x\sim P_x}\sup_{S_\delta}\| R[g](A x+\varepsilon)-A^\dagger A x\|^2\\
        \leq& \min_{g}\mathbb{E}_{x\sim P_x}2\left(\| R[g](A x)-A^\dagger A x\|^2 + \sup_{S_\delta} \sum_n g_n^2 \langle\varepsilon, v_n\rangle^2\right)\\
        =& \min_{g}  \mathbb{E}_{x\sim P_x} 2\left(\| R[g](A x)-A^\dagger A x\|^2 +\delta^2 \|g\|_{\ell^2}^2\right)\\
        \leq& \mathbb{E}_{x\sim P_x} 2\left(\| R[g(1/\sqrt{\delta})](A x)-A^\dagger A x\|^2 +\delta\right) \rightarrow 0
    \end{align*}
    as $\delta \rightarrow 0$. This implies that in particular $\sum_n (g_n^{adv_\infty})^2 \delta^2 \rightarrow 0$. Then for a general distribution $\tilde{P}_x$
    \begin{align*}
    &\mathbb{E}_{x\sim\tilde{P}_x} \sup_{S_\delta}\| R_{g^\infty(\delta)}(A x+\varepsilon)-A^\dagger Ax\|^2\\
    \leq& 2 \sum_{n} (1-\sigma_n g_n^{adv_\infty})^2 \mathbb{E}_{x\sim \tilde{P}_x} \langle x,u_n\rangle^2 +(g_n^{adv_\infty})^2 \delta^2\\
    \leq&2\sum_{\delta \geq \mathbb{E}_{x\sim\tilde{P}_x} |\langle x,u_n\rangle|\sigma_n } \mathbb{E}_{x\sim \tilde{P}_x}\langle x, u_n\rangle^2+2\sum_n (g_n^{adv_\infty})^2 \delta^2
    \end{align*}
    The first term in the last line clearly vanishes as $\delta \rightarrow 0$ since $\sum_n \mathbb{E}_{x\sim\tilde{P}_x} \langle x,u_n\rangle ^2$ is finite and the second term vanishes due to previous considerations.
\end{proof}
\section{Sparse Frame Architectures}

In order to obtain further insights into more realistic architectures used in deep learning, we construct an example of a network-based regularizer that resembles sparsity in frame systems, such as analysis-based formulations of wavelet regularization or the heavily used total variation regularization. In order to learn a regularizer we follow the setup of adversarial regularizers proposed in \cite{lunz2018adversarial_MB} with a data set of random ellipses, the forward operator $A$ being the identity, which is augmented by additive Gaussian noise. The adversarial examples are thus noisy images as illustrated in \ref{fig:TVdata_MB}.

\begin{figure}
\center\includegraphics[width = 11cm]{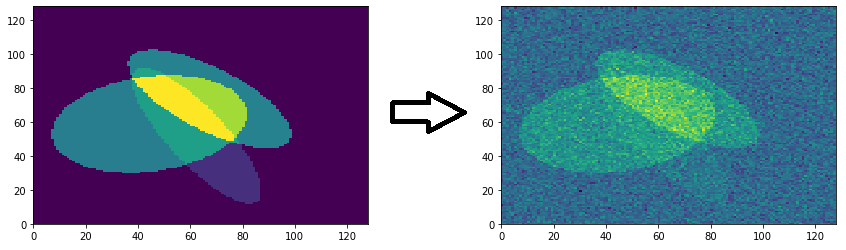}
\caption{Illustration of typical data (clean images on the left) and adversarial data (noisy images on the right). \label{fig:TVdata_MB}}
\end{figure}

\subsection{Sparse CNN Architectures}

In order to mimic sparse multiresolution analysis such as in wavelets, we use three convolutional kernels ${\mathbf k}=(k^1,k^2,k^3 )$ of different sizes. 
As a regularizer we choose
\begin{equation}
R(x,{\mathbf k} ) = \Vert {\mathbf k} \ast x \Vert_{1,1} = \sum_{i=1}^3  \Vert k^i \ast x \Vert_{1}.
\end{equation}
The three kernels are of size $6 \times 6$, $16 \times 16$, $36 \times 36$, and stride $2, 4$ and $8$, respectively, i.e. they represent different resolutions from fine to coarse.\\

\begin{figure}
\center\includegraphics[width = 10.5cm]{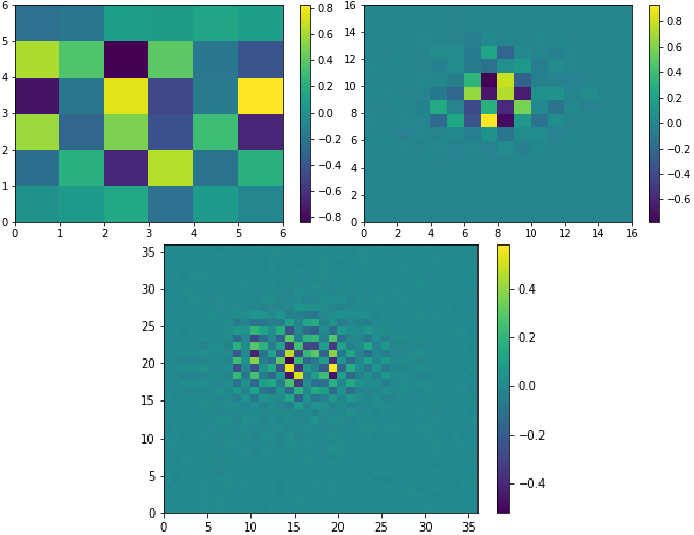}
\caption{Learned kernels $k^1,k^2,k^3$ in a multiresolution CNN architecture.. \label{fig:CNNkernelsrand_MB}}
\end{figure}

The results after training as an adversarial regularizer (using a random initialization) are shown in Figure \ref{fig:CNNkernelsrand_MB}. Surprisingly, all kernels rather have an effective scale of $6x6$ after removing  approximately vanishing entries. This indicates that a local CNN architecture is more suitable than multiple resolutions. Let us also mention that resulting kernels we obtained depend a lot on the initialization used for training, e.g. they look different when using only constant kernels initially. However, the main effect of all kernels having local support appears to be consistent also with other initializations

\subsection{Do local CNNs result in  Total Variation ?}

In order to further understand the behavior of local kernels in convolutional neural networks, we investigate a network architecture consisting of four convolutional kernels ${\mathbf k}=(k^1,k^2,k^3,k^4)$ of size $2\times 2$ and define the regularizer as 
\begin{equation}
R(x,{\mathbf k} ) = \Vert {\mathbf k} \ast x \Vert_{2,1} .
\end{equation}
This regularizer architecture resembles isotropic total variation in the sense that first the Euclidean norm of the four convolutions with the image at each point is computed (similar to the Euclidean norm of derivatives) and then summed up over all pixels. 

\begin{figure}
\center\includegraphics[width =10.0cm]{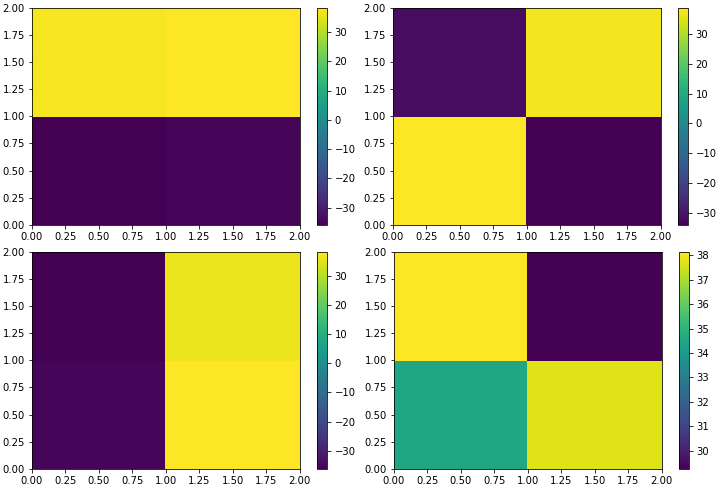}
\caption{Learned kernels $k^1,k^2,k^3,k^4$ in a local CNN architecture. \label{fig:CNNkernelslearned_MB}}
\end{figure}

Figure \ref{fig:CNNkernelslearned_MB} displays the four kernels obtained after training. Indeed, we see that the regularizer resembles a total variation type functional. The two kernels on the left can be interpreted exactly as realizations of derivatives along the coordinate axis. The kernels on the right can also be interpreted very well:  The lower one is almost constant (note the color bar) and hence penalizes the mean value. The top one resembles the approximation of a mixed second derivative, which can be seen easily from a Taylor expansion. The comparison of absolute values in all kernels, which are of similar range, implies that indeed the mean value part is effectively weighted stronger than the first derivatives (effectively divided by the number of pixels in one direction) and much stronger than the mixed second derivative (effectively divided by the total number of pixels) .

\section*{Acknowledgement}
MB, SK and LW acknowledge support from DESY (Hamburg, Germany), a member of the Helmholtz Association HGF. MB and SK acknowledge support from the German Research Foundation, project BU 2327/19-1. MB and LW acknowledge support from the German Research Foundation, project BU 2327/20-1. GK acknowledges support from the German Research Foundation, project KU 1446/32-1.

\printbibliography

\clearpage
\section*{Appendix A}
\addcontentsline{toc}{section}{Appendix A}
To prove Theorem \ref{mb_thm:convratedecayrates}, we are mainly interested in the decay of  $\sum_{n = N+1}^\infty n^{-a}$ and the growth of 
    $\sum_{n = 1}^N n^{b}$ as $N$ grows. We summarize the estimates we need in the following lemma.
    
    \begin{lemma}\label{mb_lem:hilfslemmadecay}
    Let $N > 0$.
        For $a > 1$, it holds that
        \begin{equation}\label{mb_eq:smalltail}
            \sum_{n = N+1}^\infty n^{-a} \leq \frac{a}{a-1}(N+1)^{1-a}.
        \end{equation}
        For $0 \geq b > -1$, {  we have }
        \addtocounter{equation}{1}
        \begin{equation}\label{mb_eq:finitesumsmalltail}\tag{\theequation a}
            \sum_{n = 1}^N n^{b} \leq \frac{1}{1+b} N^{1+b},
        \end{equation}
        and, for $b > 0$,
        \begin{equation}\label{mb_eq:finitesumlargetail}\tag{\theequation b}
             \sum_{n = 1}^N n^{b} \leq N^{1+b}.
        \end{equation}
     \end{lemma}
     \begin{proof}
         We start with the proof of \eqref{mb_eq:smalltail}. For $a > 1$, the function $\R \ni x \mapsto x^{-a}$ is decreasing in $x$ and therefore
         \begin{equation*}
             \sum_{n = N+2}^\infty n^{-a} \leq \int_{N+1}^\infty x^{-a}\, dx = \frac{1}{a-1} (N+1)^{1-a}, 
         \end{equation*}
         for any $N \in \N$. Additionally, we use that $(N+1)^{-a} \leq (N+1)^{1-a}$ for any $N \in \N$ to show that
         \begin{equation*}
             \sum_{n = N+1}^\infty n^{-a} = (N+1)^{-a} + \sum_{n = N+2}^\infty n^{-a} \leq \left(1 + \frac{1}{a-1}\right) (N+1)^{1-a} = \frac{a}{a-1} (N+1)^{1-a}.
         \end{equation*}
         The claim \eqref{mb_eq:finitesumsmalltail} with $0 \geq b > -1$ follows from the analogue integral bound using that $\R \ni x \mapsto x^{b}$ is not increasing in $x$ and therefore
         \begin{equation*}
             \sum_{n = 2}^N n^{b} \leq \int_{1}^N x^b = \frac{1}{1+b}\left(N^{1+b}-1\right).
         \end{equation*}
         Since further $1/(1+b) \geq 1$ this yields
         \begin{equation*}
             \sum_{n = 1}^N n^{b} = 1 + \sum_{n = 2}^N n^{b} \leq \frac{1}{1+b} N^{1+b}.
         \end{equation*}
         The bound \eqref{mb_eq:finitesumlargetail} is a direct consequence of $n^b \leq N^{b}$ for all $n = 1,\hdots,N$. 
     \end{proof}
    With Lemma \ref{mb_lem:hilfslemmadecay} we can show the decay rate for fixed noise distributions, which directly implies Theorem \ref{mb_thm:convratedecayrates}.
     \begin{lemma}\label{mb_lem:convratedoubleexp}
         Let Assumption \ref{mb_assmp:decayrates} hold, i.e., $\Pi_n = O(n^{-a})$ with $a > 1$ and $\Delta_n/\sigma_n^2 = O(\delta^2 n^{b})$ with $b \in \R$. If $b > -1$, it holds that
         \begin{equation*}
             \sum_{n \in \N} \frac{\Pi_n \, \Delta_n}{\Pi_n \sigma_n^2 + \Delta_n} \lesssim \delta^{2\frac{a-1}{a + b}},
         \end{equation*}
         where the constant depends on $a$ and $b$.
         If $b < -1$, {  we have}
         \begin{equation*}
             \sum_{n \in \N} \frac{\Pi_n \, \Delta_n}{\Pi_n \sigma_n^2 + \Delta_n} \lesssim \delta^{2},
         \end{equation*}
         where the constant depends on $b$.
     \end{lemma}
     \begin{proof}
     We start with the case $b > -1$.
        Inserting Assumption \ref{mb_assmp:decayrates} into \eqref{mb_eq:splitsum}, we get
         \begin{equation*}
             \sum_{n \in \N} \frac{\Pi_n \, \Delta_n}{\Pi_n \sigma_n^2 + \Delta_n} \lesssim  
        \delta^2\,\sum_{n = 1}^N n^{b} + \sum_{n = N+1}^\infty n^{-a} \lesssim \delta^2\,N^{1+b} + (N+1)^{1-a},
         \end{equation*}
     where the last bound follows from Lemma \ref{mb_lem:hilfslemmadecay}.
         Since the first summand is decreasing in $N$ and the second summand is increasing in $N$, we can estimate
         \begin{equation*}
             \delta^2\,N^{1+b} + (N+1)^{1-a} \leq \delta^2\,x^{1+b} + x^{1-a}
         \end{equation*}
         for any $x \in [N, N+1]$, $N \in \N$. For $\delta$ small enough we can choose $N > 1$ such that 
         \begin{equation*}
             \delta^{-\frac{2}{a+b}} \in [N, N+1]
         \end{equation*}
         and therefore 
         \begin{equation*}
             \sum_{n \in \N} \frac{\Pi_n \, \Delta_n}{\Pi_n \sigma_n^2 + \Delta_n} \lesssim \delta^2\,\,\delta^{-2\frac{1+b}{a+b}} + \delta^{2\frac{a-1}{a+b}} = 2 \,\delta^{2\frac{a-1}{a+b}}.
         \end{equation*}
         Absorbing the factor $2$ into the proportionality constant, this proves the claim for $b > -1$.
         For $b < -1$, Assumption \ref{mb_assmp:decayrates} yields that $\sum_{n \in \N} \Delta_n/\sigma_n^2 \lesssim \delta^2$, thus the decay rate follows directly from \eqref{mb_eq:deltabound}. 
     \end{proof}

\section*{Appendix B}
\addcontentsline{toc}{section}{Appendix B}
The proof of Theorem \ref{mb_thm:convratesourcecondition} is a direct consequence of the following Lemma.
     \begin{lemma}\label{mb_lem:decaysourcecondition}
         Let Assumption \ref{mb_assmp:sourcecondition} hold. Then it holds that 
         \begin{align*}
             \sum_{n \in \N} \frac{\Pi_n \, \Delta_n}{\Pi_n \sigma_n^2 + \Delta_n} \lesssim \delta^{2\frac{2\mu}{1+2\mu}},
         \end{align*}
         where the constant is given by \eqref{mb_eq:noisedatasum}.
     \end{lemma}
     \begin{proof}
     
         For $n \in \N$ and $\nu = 2\mu$ we can write
         \begin{equation*}
             \frac{\Pi_n \, \Delta_n}{\Pi_n \sigma_n^2 + \Delta_n} = \frac{\Pi_n \, \Delta_n}{\left(\Pi_n \sigma_n^2 + \Delta_n\right)^{\frac{\nu}{1+\nu}+\frac{1}{1+\nu}}} \leq \frac{\Pi_n\,\Delta_n}{\left(\Pi_n \sigma_n^2\right)^{\frac{\nu}{1+\nu}}\,\Delta_n^{\frac{1}{1+\nu}}}.
         \end{equation*}
         Inserting the assumption on $\Pi_n$ and $\Delta_n$ this yields
         \begin{equation*}
             \frac{\Pi_n \, \Delta_n}{\Pi_n \sigma_n^2 + \Delta_n} \leq \frac{\sigma_n^{2\nu}\beta_n\,\Delta_n}{\sigma_n^{2\nu} \beta_n^{\frac{\nu}{1+\nu}}\,\Delta_n^{\frac{1}{1+\nu}}} = \beta_n^{\frac{1}{1+\nu}}\frac{\delta^2 \gamma_n}{\delta^{2\frac{1}{1+\nu}}\gamma_n^{\frac{1}{1+\nu}}} = \delta^{2\frac{\nu}{1+\nu}}\, \beta_n^{\frac{1}{1+\nu}} \gamma_n^{\frac{\nu}{1+\nu}}.
         \end{equation*}
         The claim now follows from inserting $\nu = 2\mu$ and \eqref{mb_eq:noisedatasum}.
            \end{proof}

\section*{Appendix C}
\addcontentsline{toc}{section}{Appendix C}
This sections give a short proof of Lemma \ref{lm:fR_MB} and Lemma \ref{lm:SL_MB} . For this we need the following observation.
\begin{lemma}\label{lm: H1_MB}
Let $d\in l^1$. Then
\begin{align*}
    f: g\in l^\infty\mapsto \sum_n g_n^2 d_n \in \R 
\end{align*}
is weak* lower semicontinuous.
\end{lemma}
\begin{proof}
The function $x \mapsto x^2$ is convex and thus there exist  sequences $a_i$ and $b_i$ in $\R$ such that $x^2=\sup_i a_i x +b_i$ for all $x\in \R$.  We denote by $\Gamma$ the set of all sequences $m_n$ with values in $\N$ such that $\|m\|_\infty \leq C$ for some $C\in \N$. Then
clearly $$ f(g)\geq \sup_{m\in \Gamma} \sum_n (a_{m_n}g_n+b_{m_n}) d_n,$$
but in particular equality holds.
To see this we can choose a $N$ such that $\sum_{n>N} g_n^2 d_n < \frac{\varepsilon}{3}$. Then we choose the first $N$ components of the sequence $m$ such that $\sum_{1\leq n\leq N} g_n^2 d_n - (a_{m_n} g_n + b_{m_n})d_n \leq \frac{\varepsilon}{3}$. Then
$$0\leq f(g)- \sum_n (a_{m_n}g_n+b_{m_n}) d_n \leq \varepsilon.$$
Since for any $m\in\Gamma$ we have $\|m\|_\infty$, the sequence $a_{m_n}$ is bounded as well
and
$$ \Tilde{f}_m(g)=\sum_n (a_{m_n} g_n +b_{m_n}) d_n $$ is weak* lower semi continuous.
The weak* lower semicontinuity of $f$ then follows from the fact that
taking the supremum preserves weak* lower semicontinuity.
\end{proof}

\begin{proof}[ Proof of Lemma \ref{lm:fR_MB}]
The continuity in $x$ and $\varepsilon$ follows from the fact that $A^\dagger A$ , $A$ and $R[g]$ are continuous Operators. 
For the weak* lower semi continuity we write
\begin{align*}
    &\|R[g](Ax+\varepsilon) -A^\dagger Ax\|^2\\= &\sum_n  (1-\sigma_n g_n)^2\langle x, u_n\rangle^2 +g_n^2\langle \varepsilon,v_n\rangle^2
    -2(1-\sigma_n g_n)g_n\langle x, u_n\rangle\langle \varepsilon,v_n\rangle\\
    = & \sum_n \underbrace{\langle x,u_n \rangle^2}_{\eqqcolon a_n} + g_n\underbrace{\left(-2 \sigma_n\langle x,u_n\rangle^2 -2 \langle x,u_n\rangle \langle\varepsilon,v_n\rangle\right)}_{\eqqcolon b_n}\\
    & \qquad  +g_n^2 \underbrace{\left(\sigma_n^2\langle x,u_n\rangle^2+2\sigma_n \langle x,u_n\rangle \langle\varepsilon,v_n\rangle + \langle\varepsilon,v_n\rangle^2 \right)}_{\eqqcolon c_n}.
\end{align*}
Since the sequences $(a_n)_{n\in\N}, (b_n)_{n\in\N}, (c_n)_{n\in\N}$ are in $l^1$ we immediately obtain weak* continuity for the constant and linear part. For the squared part we use Lemma \ref{lm: H1_MB} to obtain weak* lower semicontinouity continuity. Since $\title{f}$ is defined as a supremum over functions the lower semicontinuity $f$ is immediately transferred to $\tilde{f}$ implying that $\tilde{f}$ is weak* lower semicontinuous in $g$ and lower semicontinuous in $x$. The lower semicontinuous in $x$ yields the Borel measurability in $x$.
Lastly for a sequence $ (g_m)_{n\in\N} \subset l^\infty$ weakly* converging to $g\in l^\infty$ we obtain by Fatou's lemma
\begin{align*}
    \liminf_{m\rightarrow \infty} \mathbb{E}_{x\sim P_x} \sup_{\varepsilon \leq \delta} \| R[g_m](Ax+\varepsilon)-A^\dagger Ax\|^2 \\
    \geq \mathbb{E}_{x\sim P_x} \liminf_{m\rightarrow \infty} \sup_{\varepsilon \leq \delta} \| R[g_m](Ax+\varepsilon)-A^\dagger Ax\|^2\\
    \geq \mathbb{E}_{x\sim P_x} \sup_{\varepsilon \leq \delta} \| R[g](Ax+\varepsilon)-A^\dagger Ax\|^2.
\end{align*}
\end{proof}

\begin{proof}[Proof of Lemma \ref{lm:SL_MB}]
    We estimate
    \begin{align*}
        \sup_{\|\varepsilon\| \leq \delta} \| R[g](Ax+\varepsilon)-A^\dagger Ax\|^2\\
        =\sup_{\|\varepsilon\| \leq \delta} \sum_n \underbrace{(1-\sigma_n g_n)^2 \langle x, u_n\rangle^2}_{\geq0} \underbrace{-2 (1-\sigma_n g_n) g_n\langle x,u_n\rangle \langle \varepsilon, v_n\rangle}_{\eqqcolon r_m} +g_n^2 \langle \varepsilon, v_n\rangle^2\\
        \stackrel{(*)}{\geq} \sup_{\|\varepsilon\| \leq \delta} \sum_n g_n^2 \langle \varepsilon, v_n \rangle^2= \delta^2 \|g\|_\infty^2.
    \end{align*}
    To see $(*)$ we simply observe that in order to approximate the first supremum the 
     $\sign$ of $\langle \varepsilon, v_n\rangle$ should be chosen such that the term $r_m$ is non negative. Therefore 
     \begin{align*}
         \mathbb{E}_{x\sim P_x} \sup_{\|\varepsilon\| \leq \delta} \| R[g](Ax+\varepsilon)-A^\dagger Ax\|^2 \geq \delta^2 \|g\|^2_\infty
     \end{align*}
     and sublevels are by Banach-Alaoglu theorem sequentially weakly* compact.  
\end{proof}

\section*{Appendix D}
\addcontentsline{toc}{section}{Appendix C}
\begin{proof}[ Proof of Theorem \ref{thm: AdInftyVal}]
    Inserting for each $x\in X$ the corresponding $w^{max}$ into $f(x,w)$, the $g_n^{adv_\infty}$ have to minimize the term
\begin{align*}
        \sum_n\underbrace{\mathbb{E}_{x\sim P_x} (1-\sigma_n g_n)^2 \langle x , u_n\rangle^2 +2\delta|1-\sigma_n g_n| |g_n| |\langle x ,u_n\rangle| + g_n^2 \delta^2}_{\eqqcolon S_n(g_n)\geq 0}.
    \end{align*}
Each summand $S_n(g_n)$ is is non-negative and can thus be optimized independently. We observe that $\mathbb{E}_{x\sim P_x}|\langle x, u_n\rangle|=0$ if and only if $\mathbb{E}_{x\sim P_x}\langle x, u_n\rangle^2=0$. In this case the optimal $g_n$ is given by $0$ and we exclude this case in the following considerations.\\
Keeping Lemma \ref{lm: GRange_MB} in mind, the $g_n$ of minimizers have to lie in the interval $\left[0,\frac{1}{\sigma_n}\right]$. We thus need to check the two boundary points $0$ and $\frac{1}{\sigma_n}$ as well as possible critical points in the interior of the interval itself. The summand $S_n$ is differentiable in the open interval $\left(0,\frac{1}{\sigma_n}\right)$ and by differentiation we obtain the optimality conditions
\begin{align*}
    \frac{d S_n}{d g_n}=&-2 \sigma_n(1-\sigma_n g_n) \mathbb{E}_{x\sim P_x}\langle x, u_n\rangle^2-2 \delta\sigma_n g_n \mathbb{E}_{x\sim P_x}|\langle x, u_n\rangle|\\
    &+2\delta(1-\sigma_n g_n)\mathbb{E}_{x\sim P_x}|\langle x, u_n\rangle|+2g_n \delta^2=0.
\end{align*}
Then a possible extreme point is given by 
\begin{align*}
g_n&=\frac{\sigma_n\mathbb{E}_{x\sim P_x}\langle x,u_n\rangle^2-\delta \mathbb{E}_{x\sim P_x}|\langle x,u_n\rangle|}{\sigma_n^2\mathbb{E}_{x\sim P_x}\langle x,u_n\rangle^2-2\delta \sigma_n \mathbb{E}_{x\sim P_x}|\langle x,u_n\rangle|+\delta^2}\\
&=\frac{1}{\sigma_n-\delta \frac{\sigma_n \mathbb{E}_{x\sim P_x} |\langle x,u_n\rangle|-\delta}{\sigma_n \mathbb{E}_{x\sim P_x} \langle x,u_n\rangle^2-\delta \mathbb{E}_{x\sim P_x} |\langle x,u_n\rangle|}}.
\end{align*}
For a better understanding of this critical point we have to study the denominator 
\begin{align*}
    D\coloneqq\sigma_n -\underbrace{\frac{\delta}{\mathbb{E}_{x\sim P_x} |\langle x,u_n\rangle|}}_{\coloneqq A}\cdot\underbrace{ \frac{\sigma_n \mathbb{E}_{x\sim P_x} |\langle x,u_n\rangle| -\delta }{\sigma_n \frac{\mathbb{E}_{x\sim P_x} \langle x,u_n\rangle^2}{\mathbb{E}_{x\sim P_x} |\langle x,u_n\rangle|}-\delta}}_{\coloneqq B}
\end{align*}
and use the flowing estimates. By Cauchy--Schwartz inequality 
\begin{align*}
    \mathbb{E}_{x\sim P_x} |\langle x,u_n\rangle|\leq\frac{\mathbb{E}_{x\sim P_x} \langle x,u_n\rangle^2}{\mathbb{E}_{x\sim P_x} |\langle x,u_n\rangle|}
\end{align*}
and thus 
\begin{align}\label{eq: CS_MB}
    \sigma_n \mathbb{E}_{x\sim P_x} |\langle x,u_n\rangle| -\delta \leq \sigma_n \frac{\mathbb{E}_{x\sim P_x} \langle x,u_n\rangle^2}{\mathbb{E}_{x\sim P_x} |\langle x,u_n\rangle|}-\delta.
\end{align}
We now have to distinguish three cases:\\[3pt]
\textbf{ Case $\frac{\delta}{\sigma_n}\leq\mathbb{E}_{x\sim P_x} |\langle x,u_n\rangle|$:}
It immediately follows that $0\leq A\leq \sigma_n$ and that the enumerator of $B$ is non-negative. Then by \eqref{eq: CS_MB} it follows $0\leq B\leq 1$ and thus $\frac{1}{D}\geq \frac{1}{\sigma_n}$. We conclude that a minimizer has to lie on the boundary.
We observe that we can estimate 
\begin{align*}
     S_n(g_n)&=(1-\sigma_n g_n)^2 \mathbb{E}_{x\sim P_x}\langle x , u_n\rangle^2 +\delta(1-\sigma_n g_n)g_n \mathbb{E}_{x\sim P_x}|\langle x ,u_n\rangle| + g_n^2 \delta^2\\
     &\geq (1-\sigma_n g_n)^2 \left(\mathbb{E}_{x\sim P_x}|\langle x , u_n\rangle|\right)^2 +\delta(1-\sigma_n g_n)g_n \mathbb{E}_{x\sim P_x}|\langle x ,u_n\rangle| + g_n^2 \delta^2\\
     &\geq (1-\sigma_n g_n)^2 \frac{\delta^2}{\sigma^2}+\delta (1-\sigma_n g_n) \frac{\delta}{\sigma_n}+\delta^2 g_n^2=\delta^2\frac{1}{\sigma_n^2}.
\end{align*}
This minimal bound is obtained if $g_n^{adv_\infty}=\frac{1}{\sigma_n}$ is chosen.\\[3pt]
\textbf{Case $\mathbb{E}_{x\sim P_x} |\langle x,u_n\rangle|< \frac{\delta}{\sigma_n}< \frac{\mathbb{E}_{x\sim P_x} \langle x, u_n \rangle^2}{\mathbb{E}_{x\sim P_x} |\langle x, u_n \rangle|}$:} It immediately follows that the enumerator of $B$ is negative and its denominator positive, thus $B<0$. Since $A\geq 0$ it follows that $\frac{1}{D}\in (0, \frac{1}{\sigma_n})$. To see that this critical point is truly a minimum we calculate the second derivative with respect to $g_n$ and estimate
\begin{align*}
    \frac{d^2S_n}{d g_n^2}(g_n)&= 2\sigma_n^2\mathbb{E}_{x\sim P_x}\langle x, u_n\rangle^2 -4 \delta \sigma_n \mathbb{E}_{x\sim P_x}|\langle x,u_n\rangle| +2\delta^2 \\
    &> 2 \sigma_n \delta \mathbb{E}_{x\sim P_x}|\langle x, u_n\rangle|-4 \delta \sigma_n \mathbb{E}_{x\sim P_x}|\langle x, u_n\rangle| +2 \sigma_n \delta \mathbb{E}_{x\sim P_x}|\langle x, u_n\rangle|=0.
\end{align*}
Thus the critical point is the true minimum.\\[3pt]
\textbf{Case $\mathbb{E}_{x\sim P_x} |\langle x,u_n\rangle|\leq\frac{\mathbb{E}_{x\sim P_x} \langle x, u_n \rangle^2}{\mathbb{E}_{x\sim P_x} |\langle x, u_n \rangle|}\leq\frac{\delta}{\sigma_n}$:} It immediately follows that $A\geq \sigma_n$ and that the denominator of $B$ is non-positive. Then by \eqref{eq: CS_MB} it follows $B\geq 1$ and thus $\frac{1}{D}\leq 0$.
We conclude that a minimizer has to lie on the boundary.
We observe that we can estimate
\begin{align*}
    S_n(g_n)=&(1-\sigma_n g_n)^2 \mathbb{E}_{x\sim P_x}\langle x , u_n\rangle^2 +\delta(1-\sigma_n g_n)g_n \mathbb{E}_{x\sim P_x}|\langle x ,u_n\rangle| + g_n^2 \delta^2\\
    \geq& (1-\sigma_n g_n)^2 \mathbb{E}_{x\sim P_x}\langle x , u_n\rangle^2 +(1-\sigma_n g_n) \sigma_n g_n \mathbb{E}_{x\sim P_x}\langle x ,u_n\rangle^2\\
    &+ g_n^2 \sigma_n^2 \left(\frac{\mathbb{E}_{x\sim P_x}\langle x , u_n\rangle^2}{\mathbb{E}_{x\sim P_x}|\langle x , u_n\rangle|} \right)^2\\
    \geq& \mathbb{E}_{x\sim P_x}\langle x , u_n\rangle^2 .
\end{align*}
This minimal bound is obtained if $g_n^{adv_\infty}=0$ is chosen.
\end{proof}
\end{document}